\documentclass[a4paper,11pt]{article}
\usepackage[margin=2.5cm]{geometry}
\usepackage{amsfonts}
\usepackage[T1]{fontenc}
\usepackage{microtype}
\usepackage{makeidx}
\usepackage{setspace}
\usepackage{authblk}
\usepackage{epigraph}
\usepackage{graphicx}
\usepackage{xcolor}
\usepackage[all]{xy}
\usepackage{amsmath}
\usepackage{amssymb}
\usepackage{booktabs}
\usepackage{braket}
\usepackage{array}
\usepackage{tabularx}
\usepackage{multirow}
\usepackage{indentfirst}
\usepackage{cite}
\usepackage{stmaryrd}
\usepackage{faktor}
\usepackage{titling}
\usepackage{tikz}
\usepackage{dsfont}
\usepackage{url}
\usepackage{hyperref}
\usepackage{tikz-cd}
\usepackage{tkz-graph}

\usetikzlibrary{matrix,arrows,decorations.pathmorphing}
    \renewcommand{\leq}{\leqslant}
    \renewcommand{\geq}{\geqslant}
\usepackage{amsthm}
\theoremstyle{plain}
\newtheorem{thm}{Theorem}[section]
\newtheorem{dfn}[thm]{Definition}

\newtheorem{prop}[thm]{Proposition}
\newtheorem{cor}[thm]{Corollary}

\theoremstyle{definition}
\newtheorem{ex}[thm]{Example}

\theoremstyle{remark}
\newtheorem{oss}[thm]{Remark}

\DeclareMathOperator{\Lh}{\preccurlyeq_{\mathrm{L}}}

\DeclareMathOperator{\imm}{\mathrm{Im}}

\DeclareMathOperator{\Pal}{\mathrm{Pal}}

\DeclareMathOperator{\N}{\mathbb{N}}
\DeclareMathOperator{\Z}{\mathbb{Z}}

\DeclareMathOperator{\Sm}{\mathrm{Sm}}

\DeclareMathOperator{\lex}{\mathrm{lex}}

\DeclareMathOperator{\SSS}{\mathbb{S}}
\DeclareMathOperator{\BBB}{\mathbb{B}}

\title{\bf{The Lehmer complex of a Bruhat interval}}
\author{Davide Bolognini\thanks{Dipartimento di Ingegneria Industriale e Scienze Matematiche, Università Politecnica delle Marche, Ancona, Italy.
\href{mailto:davide.bolognini.cast@gmail.com}{d.bolognini@univpm.it} } \ and Paolo Sentinelli\thanks{ Dipartimento di Matematica, Politecnico di Milano, Milan, Italy. \\ \href{mailto:paolosentinelli@gmail.com}{paolo.sentinelli@protonmail.com}}}
\date{}
\begin{document}

\maketitle

\vspace{-3em}

    \epigraph{Quest'è una ragna \\ dove il tuo cuor \\ casca, si lagna, \\ s'impiglia e muor.}{Arrigo Boito, \emph{Otello}
    }

\vspace{1em}

\begin{abstract}
We introduce Lehmer codes for several finite Coxeter groups, including all classical Weyl groups. These codes correspond to embed products of chains in the Bruhat order, which allow us to associate to each lower Bruhat interval of these groups a multicomplex whose $f$-polynomial is the Poincaré polynomial of the interval. Via a general construction, we prove that these polynomials are $h$-polynomials of vertex-decomposable simplicial complexes, which we call Lehmer complexes. 
Moreover we provide a classification, in terms of unimodal permutations, of Poincaré polynomials of smooth Schubert varieties in flag manifolds.
\end{abstract}
 
\section{Introduction}
In this paper we introduce the {\em Lehmer complex} of a lower Bruhat interval in a finite Coxeter system. This is a shellable simplicial complex whose $h$-polynomial is the rank-generating function of the interval. For Weyl groups, Bj\"orner and Ekedahl in \cite{bjorner} studied these polynomials by exploiting the geometry of the underlying Schubert varieties. We develop purely combinatorial tools to deal with the shape of lower Bruhat intervals, which in principle could be applied to any finite Coxeter group.

Our definition of Lehmer complexes relies on the notion of {\em Lehmer code} for a finite Coxeter group and on a general construction associating to a multicomplex a shellable simplicial complex whose $h$-vector is the $f$-vector of the multicomplex. Let $(W,S)$ be a finite Coxeter system and $e_1,\ldots,e_n$ its exponents. We define a Lehmer code for $(W,S)$ as a bijection $$L: W \rightarrow [e_1+1] \times \ldots \times [e_n+1]$$ such that $L^{-1}$ is a poset morphism between the direct product of chains $[e_1+1] \times \ldots \times [e_n+1]$ and the Bruhat order $(W,\leq)$, see Definition \ref{def codici di L}. Given a lower Bruhat interval, a Lehmer code realizes a bijection between it and an order ideal of $[e_1+1] \times \ldots \times [e_n+1]$. Since order ideals of $\N^k$ are precisely multicomplexes, this allows us to consider lower Bruhat intervals as multicomplexes.

For a finite multicomplex $J$, Bj\"orner, Frankl and Stanley in \cite{bjornerstanley} introduced a shellable simplicial complex $M(J)$, which we call {\em $M$-complex}, such that the $h$-polynomial of $M(J)$ equals the $f$-polynomial of $J$, see our Definition \ref{definizione gei}. Given a Lehmer code $L$ for $(W,S)$ and  $w \in W$, the Lehmer complex $L_w(W,S)$ is the $M$-complex of the multicomplex $\{L(v)  : v \leq w\}$, see Definition \ref{def per intervalli}. 
The $h$-polynomial of $L_w(W,S)$ equals the rank-generating function of the interval $[e,w]:=\{v \in W  : v \leq w\}$, where $e$ is the identity of the group; this is stated in Theorem \ref{teorema h pol ranghi}. In Definition \ref{def complesso2} we introduce a Lehmer complex for any finite Coxeter group; indeed, to realize such a complex for the whole group we do not need a Lehmer code. With respect to the length statistic, this is the analog of the Coxeter complex, whose $h$-polynomial is the generating function of the descent statistic.

In Theorem \ref{teorema maduro} we provide the following formula for the rank-generating function $h_w(q)$ of a lower Bruhat interval $[e,w]$ in terms of $q$-analogs depending on the maxima of the associated multicomplex:
$$\sum_{v\leq w} q^{\ell(v)}=\sum\limits_{\varnothing \neq X\subseteq \max\{L(v): v\leq w\}}(-1)^{|X|+1}\prod\limits_{i=1}^{|S|}[(\wedge_{x\in X}\,x)_i+1]_q.$$
By this result, we believe it is an important task to compute the maxima of the order ideal $\{L(v)  : v \leq w\}$ directly from the element $w$, without complete knowledge of the interval. 

The previous results hold whenever a Lehmer code for a finite Coxeter system is available. The Lehmer code for a symmetric group $S_n$ is a well-known bijection between the group and $[2] \times \ldots \times [n]$ described in terms of inversions and it is involved in the combinatorics of Schubert polynomials and their generalizations, see e.g. the paper of Fulton \cite{Fulton}. This notion is compatible with the Bruhat order, in the sense of Definition \ref{def codici di L} and this motivates our nomenclature. For groups of type $B_n$ and $D_n$, Lehmer codes were foreseen, although not explicitly, by Stanley in \cite[Section 7]{Stanley}. In Section \ref{sezione codici espliciti} we realize explicit Lehmer codes for types $A_n$, $B_n$, $D_n$ and $H_3$ by exploiting the Coxeter structure of the groups of these types. For these classes of Coxeter systems, by the fact that $M$-sequences coincide with $f$-vectors of multicomplexes, in Proposition \ref{cor biorner} we recover part of \cite[Theorem E]{bjorner}, extending it to type $H_3$. The existence of Lehmer codes for other finite irreducible Coxeter systems remains an open problem.

In Section \ref{sezione complessi} we study $M$-complexes, collecting the tools we need to prove the previous results. In Corollary \ref{corollario ideali}, we prove that every linear extension of the order ideal $J$ provides a shelling of $M_d(J)$, recovering the fact that the rank-generating function of the order ideal $J$ coincides with the $h$-polynomial of $M_d(J)$. Along the way (see Proposition \ref{prop da medaglia fields}), we obtain that 
\begin{center}
$\{\mbox{$h$-vectors of flag $M$-complexes}\}=\{\mbox{$f$-vectors of flag simplicial complexes}\}.$
\end{center}
Since flag $M$-complexes are completely balanced and vertex-decomposable, this equality provides further evidence for a conjecture due to Costantinescu and Varbaro \cite[Conjecture 1.4]{Costantinescu}.

The notion of Lehmer code $L$ for a Coxeter system $(W,S)$ leads us to consider a distinguished set of elements of $W$, which in Definition \ref{def principali} we call {\em $L$-principal elements}. Principal elements for our Lehmer codes in type $A_n$ and $B_n$ were already considered under the name of Kempf elements. They arise in relation to desingularizations of Schubert varieties in Grassmannians \cite{Huneke} and to toric degenerations of Schubert varieties in flag varieties \cite{Gonciulea}. In Definition \ref{def unimodale}, we introduce {\em $L$-unimodal elements}. Proposition \ref{prop unimodale lehmer} states that the unimodal elements of the Lehmer code $L_{A_n}$ of Section \ref{sezione tipo a} are exactly unimodal permutations. In Theorem \ref{teorema lisce} we prove that unimodal permutations classify Poincaré polynomials of smooth Schubert varieties in flag manifolds, namely in $S_n$ we have
$$\left\{\sum_{v\leq w} q^{\ell(v)}:\mbox{$w$ is smooth} \right\}  
 = \left\{\sum_{v\leq w} q^{\ell(v)}: \mbox{$w$ is unimodal} \right\}.
$$
Proposition \ref{roba da annals} provides an alternative classification in terms of partitions into distinct parts. Similarly, the palindromic rank-generating functions of lower Bruhat intervals in type $H_3$ are classified by unimodal elements, specifically those associated to the Lehmer code $L_{H_3}$ defined in Section \ref{lehmer H3}, see Example \ref{es palindromi H3}. 

\section{Notation and preliminaries}
In this section we fix notation and recall some definitions useful for the rest of the paper. We drop the number $0$ from the set of natural numbers, i.e. $\N:=\{1,2,3,\ldots\}$. For $n\in \mathbb{N}$, define $[n]:=\left\{1,2,\ldots,n\right\}$. For a set $X$, we denote by $|X|$ its cardinality, by $\mathcal{P}(X)$ its power set and by $X^n$ its $n$-th power under Cartesian product. If $x\in X^n$, we denote by $x_i$ the projection of $x$ onto the $i$-th factor.
The $q$-analog of $n$ is a polynomial defined by $[n]_q:=\sum\limits_{i=0}^{n-1}q^i$.\\ 

{\bf Posets}. Let $(P,\leq)$ be a finite poset. Given $x,y \in P$ such that $x \leq y$, define the \emph{interval} $[x,y]:=\{z \in P: x \leq z \leq y\}$. A {\em linear extension} of $P$ is a total order $x_1,\ldots,x_r$ of the elements of $P$ such that $x_i \leq x_j$ implies $i \leq j$, for all $i,j\in [r]$. Given two posets $(P,\leq_P)$ and $(Q,\leq_Q)$, a function $f:P \rightarrow Q$ is a {\em poset morphism} if $x \leq_P y$ implies $f(x) \leq_Q f(y)$. The poset $(P,\leq)$ is a {\em lattice} if the induced subposets $\{z \in P: z \leq x, z \leq y\}$ and $\{z \in P: z \geq x, z \geq y\}$ have a unique maximum $x \wedge y$ and a unique minimum $x \vee y$, respectively. Recall that the Cartesian product of posets is a poset under the componentwise order. A set $J \subseteq P$ is an {\em order ideal} of $P$ if $x \in J$ and $y \leq x$ implies $y \in J$. We say that $x \leq y$ is a {\em covering relation} $x \vartriangleleft y$ if $|[x,y]|=2$. A {\em chain} $x_0<x_1<\ldots<x_n$ of rank $n$ in $P$ is {\em saturated} if $x_0 \vartriangleleft x_1 \vartriangleleft \ldots \vartriangleleft x_n$. A saturated chain $x_0<x_1<\ldots<x_n$ not contained in other chains is called {\em maximal}. If all maximal chains have the same rank $n$, the poset $P$ is said to be {\em graded} of rank $n$. In a graded poset $P$, the {\em rank function} $\rho:P \rightarrow \N \cup \{0\}$ is defined by setting $\rho(x)$ equal to the rank of the order ideal $\{y \in P: y \leq x\}$, for every $x \in P$. The {\em rank-generating function} or \emph{Poincaré polynomial} of $P$ is the polynomial $\sum_{x \in P}q^{\rho(x)}$. \\

{\bf Simplicial complexes and multicomplexes}. Let $V$ be a finite set. A {\em simplicial complex on a set of vertices $V$} is an order ideal $X \neq \varnothing$ of $(\mathcal{P}(V),\subseteq)$, whose elements are called {\em faces}. The maximal elements of $X$ are called {\em facets}. It is clear that we can identify a simplicial complex with the set of its facets.  The {\em dimension} of $X$ is $\mathrm{max}\{|F|-1: F\in X\}$. A simplicial complex is {\em pure of dimension $d$} if all its facets have the same cardinality $d+1$. 

The {\em $f$-vector} $(f_0,\ldots,f_d,f_{d+1})$ of a $d$-dimensional simplicial complex $X$ is defined by $$f_i:=|\{F \in X: |F|=i\}|,$$ for every $0 \leq i \leq d+1$. The {\em $h$-polynomial} $\sum\limits_{i=0}^{d+1}h_ix^i$ of a $d$-dimensional simplicial complex $X$ is defined by the equality

$$\sum\limits_{i=0}^{d+1}h_ix^{d+1-i}=\sum\limits_{i=0}^{d+1}f_i(x-1)^{d+1-i}.$$

The rank-generating function $\sum_{F \in X} q^{|F|}$ is the {\em $f$-polynomial} of $X$. The {\em $h$-vector} of $X$ is $(h_0,h_1,\ldots,h_{d+1})$. 

Given a vertex $v$ of $X$, recall that the {\em link} of $v$ is the simplicial complex $\mathrm{link}_X(v):=\{F \setminus \{v\}: F \in X, v \in F\}$ and the {\em deletion} of $v$ is the simplicial complex $\mathrm{del}_X(v):=\{F \in X: v \notin F\}.$

\vspace{1em}
A pure $d$-dimensional simplicial complex is
\begin{itemize}
\item {\em thin} if every face of cardinality $d$ is contained in exactly two facets. It is {\em subthin} if it is not thin and every face of cardinality $d$ is contained in at most two facets. 
\item {\em shellable} if there exists a {\em shelling order}, i.e. a linear order $F_1,\ldots,F_r$ on its facets such that for every $1 \leq i<j \leq r$ there exist $v \in F_j \setminus F_i$ and $h<j$ such that $F_j=(F_h \cap F_j) \cup \{v\}$. The {\em restriction} of a facet $F_i$ is defined by $\mathcal{R}(F_i):=\{v \in F_i: F_i \setminus \{v\} \in \bigcup_{j=1}^{i-1} \mathcal{P}(F_j)\}.$ Recall that for a shellable complex the $h$-vector satisfies $h_j=|\{i \in [r]: |\mathcal{R}(F_i)|=j\}|$, for all $0 \leq j \leq d+1$. In particular the $h$-vector of a shellable complex is non-negative.
\item {\em vertex-decomposable} if it is a simplex (possibly empty) or if there exists a vertex $v \in X$ satisfying the following properties: 
\begin{enumerate}
\item $\mathrm{del}_{X}(v)$ is pure of dimension $d$, i.e. $v$ is a {\em shedding vertex};
\item $\mathrm{link}_{X}(v)$ is vertex-decomposable;
\item $\mathrm{del}_{X}(v)$ is vertex-decomposable.
\end{enumerate}
\item {\em balanced of type $a \in (\N\cup \{0\})^{k}$}, where $k \in \N$ and $\sum_{i=1}^k a_i=d+1$, if there exists a partition $V=S_1 \cup \ldots \cup S_k$ such that for every facet $F$ of $X$ we have $|F \cap S_i|=a_i$, for every $i \in [k]$. Notice that our definition is slightly different from the definition given in \cite[Chapter 3]{StaCCA}, because we allow the case $a_i=0$. It is {\em completely balanced} if it is balanced of type $(1,\ldots,1)$.
\end{itemize}

Let $k \in \N$. A {\em multicomplex} is an order ideal of $\N^k$. A simplicial complex is a finite multicomplex, considered as order ideal of $[2]^k$. The {\em $f$-polynomial} (or $f$-series if the multicomplex is infinite) of a multicomplex is its rank-generating function. The $f$-\emph{vector} of a multicomplex is the tuple of coefficients of its $f$-polynomial and it is called {\em $M$-sequence}. A numerical characterization of $M$-sequences, which coincide with Hilbert functions of a standard graded $K$-algebra, was proved by Macaulay, see \cite[Theorem 2.2]{StaCCA}. Moreover, finite $M$-sequences coincide with $h$-vectors of pure shellable simplicial complexes, see \cite[Theorem 3.3]{StaCCA}.\\ 

{\bf Finite Coxeter systems}. A finite {\em Coxeter system} $(W,S)$ is a finite group $W$ generated by a set $S$ of involutions, and relations $(st)^{m(s,t)}=e$, for all $s,t \in S$, where $e$ is the identity in $W$ and $m: S \times S \rightarrow \mathbb{N}$ is the so-called Coxeter matrix. For the classification of finite Coxeter systems we refer to \cite[Appendix A1]{BB}. 

Given $X,Y \subseteq W$ and $w\in W$, we let $XY:=\{xy: x\in X, \, y \in Y\} \subseteq W$ and $wX:=\{wx: x\in X\} \subseteq W$. Analogously we define $Xw$.

We denote by $\ell:W \rightarrow \N \cup \{0\}$ the {\em length function}. For $X\subseteq W$, we let $X(q):=\sum_{x\in X}q^{\ell(x)}$. For any $J\subseteq S$, the subgroup generated by $J$ is denoted by $W_J$ and we set $$^{J}W:=\left\{w\in W:\ell(sw)>\ell(w),\,\forall\, s\in J\right\}.$$ Any element $w\in W$ factorizes uniquely as $w=w_J({^Jw})$, where ${^Jw}\in {^JW}$, $w_J\in W_J$ and $\ell(w)=\ell(w_J)+\ell({^Jw})$. Given $I\subseteq J \subseteq S$, we set 
$$^{I}W_J:=\left\{w\in W_J:\ell(sw)>\ell(w),\,\forall\, s\in I\right\}.$$

One of the most important features of a Coxeter group is a natural partial order $\leqslant$ on it, called {\em Bruhat order}. It can be defined by the subword property (see \cite[Chapter~2]{BB}). The Bruhat order is a refinement of the so-called {\em weak orders}, see \cite[Chapter 3]{BB}. We denote by $\leq_L$ the left weak order, and by $\leq_R$ the right one. The posets $(W,\leqslant_L)$, $(W,\leqslant_R)$ and $(W,\leqslant)$ are graded with rank function $\ell$. Their minimum is the identity $e$. We denote by $w_0$ their maximum and by $w_0(J)$ the maximum of $W_J$. The {\em exponents} of a finite Coxeter system $(W,S)$ are the uniquely determined positive integers $e_1,\ldots,e_{|S|}$ such that $W(q)=\prod_{i=1}^{|S|}[e_i+1]_q$.

The {\em symmetric group} of permutations of $n$ objects is denoted by $S_n$. A permutation $\sigma \in S_n$ can be written in one line notation as $\sigma(1)\sigma(2)\ldots\sigma(n)$. The group $S_n$ is a Coxeter group; its standard Coxeter presentation has generators $S=\left\{s_1,\ldots,s_{n-1}\right\}$, where $s_i$ is the permutation $12\ldots(i+1)i\ldots n$, for all $i\in [n-1]$. The element of maximal length is $w_0=n(n-1)\ldots21$. We say that a permutation $\sigma \in S_n$ {\em avoids the pattern} $\tau \in S_k$, with $k \leq n$, if there are no $1 \leq i_1<i_2<\ldots<i_k \leq n$ such that $\sigma(i_{\tau^{-1}(1)})<\sigma(i_{\tau^{-1}(2)})<\ldots<\sigma(i_{\tau^{-1}(k)})$. Recall that the Bruhat order on $S_n$ represents the inclusions between Schubert varieties in the flag variety. It is well-known that smooth Schubert varieties correspond to permutations avoiding the patterns $3412$ and $4231$. 

\subsection{Multicomplexes and M-complexes} \label{sezione complessi}

In this section, given a finite multicomplex, we study a vertex-decomposable simplicial complex, which we call $M$-complex, whose $h$-vector equals the $f$-vector of the multicomplex. This construction is a special case of those introduced by Bj\"orner, Frankl and Stanley in \cite[Section 5]{bjornerstanley} to prove the equivalence between finite $M$-sequences and $h$-vectors of shellable simplicial complexes. The same construction was further explored by Murai in \cite{Murai}. Our presentation is different from the previous ones and more convenient for our purposes. $M$-complexes, together with a Lehmer code when available, provide Lehmer complexes, as we show in subsequent sections. Our approach highlights the fact that the $h$-polynomial of a Lehmer complex realizes the generating function of the length statistic of a lower Bruhat interval similarly to the way in which $h$-polynomials of Coxeter complexes realize the generating functions of the descent statistic of an interval in the weak order. 

Let $k\in \N$ and $d:=(d_1,\ldots,d_k)\in \N^k$. We shall define a pure simplicial complex $M_d$ of dimension $$\mathrm{dim}(M_d)=\sum\limits_{i=1}^k d_i-k-1$$
with number of facets equal to $\prod_{i=1}^k d_i$. 
For $x\in [d_1] \times \ldots \times [d_k]$ we define a set 
\begin{equation} \label{eq faccette}
    F_x:=\bigcup\limits_{i=1}^k \left([d_i]\setminus \{d_i+1-x_i\}\right)\times \{i\}.
\end{equation}
For example, if $k=3$, $d_1=d_2=3$ and $d_3=4$, we have that
$$F_{(1,2,2)}=\{(1,1),(2,1),(1,2),(3,2),(1,3),(2,3),(4,3)\}.$$
\begin{dfn}\label{def complesso} Let $k\in \N$ and $d \in \N^k$. We call the simplicial complex $$M_d:=\left\{F_x: x \in [d_1] \times \ldots \times [d_k]\right\}$$ the \emph{$M$-complex of degrees} $d_1,\ldots,d_k$.
\end{dfn} 
Notice that if $d=(1,1,\ldots,1)$ then $M_d=\{\varnothing\}$.
We call the numbers  $e_i:=d_i-1$ the \emph{exponents} of $M_d$, for all $i\in [k]$. Clearly the simplicial complex of Definition \ref{def complesso} is pure, balanced of type $(e_1,\ldots,e_k)$, and it has the announced dimension and number of facets. 

We consider the componentwise order on the set $[d_1]\times \ldots \times [d_k]$. Its rank function is denoted by $\tilde{\rho}: [d_1]\times \ldots \times [d_k] \rightarrow \N \cup \{0\}$, i.e. $x\mapsto \tilde{\rho}(x):=\sum_{i=1}^k(x_i-1)$. Similarly to Coxeter complexes, in the next result we prove that all linear extensions of this poset are shellings for the $M$-complex $M_d$. 

\begin{thm}\label{teorema shell}
  Let $M_d$ be the $M$-complex of degrees $d_1,\ldots,d_k$ and $\omega:=\prod_{i=1}^kd_i$ the number of facets of $M_d$. If $x_1,\ldots,x_\omega$ is a linear extension of  $[d_1]\times \ldots \times [d_k]$, then $F_{x_1},\ldots,F_{x_\omega}$ is a shelling order for $M_d$.  
\end{thm}
\begin{proof} 
If $k=1$ the result is straightforward from the definition of $M_d$.  So assume $k>1$. Let $x_1,\ldots,x_\omega$ be a linear extension of $[d_1]\times \ldots \times [d_k]$; in particular $x_1=(1,1,\ldots,1)$.  Let $i,j\in [\omega]$ be such that $1\leq i <j$.
Then we have that either $x_i<x_j$ or there exist $a,b\in [k]$ such that $(x_i)_a>(x_j)_a$ and $(x_i)_b<(x_j)_b$. In both cases there exists $h\in [k]$ such that $(x_i)_h<(x_j)_h$. Notice that $(d_h+1-(x_i)_h,h)\in F_{x_j}\setminus F_{x_i}$. Let $z:=((x_j)_1,\ldots,(x_j)_{h-1},(x_i)_h,(x_j)_{h+1},\ldots, (x_j)_k)$. Since $(x_i)_h<(x_j)_h$, we have that $z<x_j$; hence there exists $r<j$ such that $z=x_r$. 
Moreover $F_{x_j}\setminus F_{x_r}=\{ (d_h+1-(x_i)_h,h) \}$. This proves our statement.    
\end{proof}

The simplicial complexes of the next definition are, mutatis mutandis, the ones introduced in \cite[Section 5]{bjornerstanley}.

\begin{dfn} \label{definizione gei}
Let $k \in \N$ and $d \in \N^k$. For any subset $\varnothing \neq J\subseteq [d_1]\times \ldots \times [d_k]$, the $M$-complex of $J$ is 
  $M_d(J):=\{F_x: x\in J\} \subseteq M_d$, where the facet $F_x$ is defined in \eqref{eq faccette}.
\end{dfn} Recall that multicomplexes are precisely order ideals of $\mathbb{N}^k$, for some $k \in \N$. Hence, finite multicomplexes are order ideals of $[d_1] \times \ldots \times [d_k]$, for some $d_1,\ldots,d_k \in \mathbb{N}$. In the next result we prove that the $h$-polynomial of $M_d(J)$ is the rank-generating function of the order ideal $J$. 

\begin{cor} \label{corollario ideali} Let $\varnothing \neq J\subseteq [d_1]\times \ldots \times [d_k]$ be an order ideal. Then any linear extension of $J$ provides a shelling order of $M_d(J)$. Moreover, the polynomial  
    $\sum_{x\in J}q^{\tilde{\rho}(x)}$ is the $h$-polynomial of the complex $M_d(J)$.
\end{cor}

\begin{proof} Since order ideals provide partial linear extensions, the first part is an immediate consequence of Theorem \ref{teorema shell}. The second part of the statement is essentially \cite[Theorem 1.14 (i)]{Murai}.\end{proof}
Let $|J|>1$. It is easy to see that the complex $M_d(J)$ is thin if and only if $J=[d_1]\times \ldots \times [d_k]$; in the other cases it is subthin. From this fact, see \cite[Appendix A2]{BB}, we deduce the following corollary, see also \cite[Lemma 1.4 (i) and (iii)]{Murai}. 
\begin{cor} \label{cor sfere}
    Let $J\subseteq [d_1]\times \ldots \times [d_k]$ be an order ideal such that $|J|>1$. Then the geometric realization of $M_d(J)$ is $PL$-homeomorphic to 
\begin{enumerate}
    \item the sphere ${\SSS}^{\mathrm{dim}(M_d)}$, if $J=[d_1]\times \ldots \times [d_k]$;
    \item the ball ${\BBB}^{\mathrm{dim}(M_d)}$, otherwise.
\end{enumerate}
\end{cor}

It is well-known that a vertex-decomposable simplicial complex is shellable. If $J \subseteq  [d_1]\times \ldots \times [d_k]$ is an order ideal, it is easy to prove that the simplicial complex $M_d(J)$ has this stronger property, see also \cite[Remark 1.8]{Murai}. 

\begin{prop} \label{teorema VD}
Let $\varnothing \neq J\subseteq [d_1]\times \ldots \times [d_k]$ be an order ideal. Then $M_d(J)$ is vertex-decomposable.
\end{prop}
By the previous results and the characterization of $M$-sequences as $f$-vectors of multicomplexes, non-zero finite $M$-sequences are exactly the $h$-vectors of $M$-complexes. More precisely, the following refinement of the well-known equivalence between finite $M$-sequences and $h$-vectors of shellable simplicial complexes holds (see \cite[Theorem 3.3]{StaCCA}). Notice that this refinement is implicit in Murai's work \cite{Murai}.

\begin{prop}\label{teo astoniscing incredibile}
Given a polynomial $P \neq 0$, the following are equivalent:
    \begin{enumerate}
        \item[$\mathrm{1)}$] the coefficients of $P$ form an $M$-sequence.
        \item[$\mathrm{2)}$] $P$ is the $h$-polynomial of a vertex-decomposable simplicial ball.
    \end{enumerate}
\end{prop}

\begin{proof}
If $P=1$ the result is trivial. So assume $P \neq 1$. 
    It is clear that $\mathrm{2)}$ implies $\mathrm{1)}$, since a vertex-decomposable simplicial complex is shellable. 
    Now we prove the other implication. Let $P$ be a polynomial whose coefficients form an $M$-sequence. It is the $f$-polynomial of a finite multicomplex, i.e. the rank-generating function of an order ideal $J \subsetneq [d_1] \times \ldots \times [d_k]$ for some $d=(d_1,\ldots,d_k) \in \mathbb{N}^k$. Consider the $M$-complex $M_d(J)$. By Corollary \ref{corollario ideali}, the $h$-polynomial of $M_d(J)$ is $P$. By construction $M_d(J)$ is balanced of type $(e_1,\ldots,e_k)$ and by Proposition \ref{teorema VD} it is vertex-decomposable. Since $P \neq 1$, $|J|>1$. By Corollary \ref{cor sfere}, $M_d(J)$ is a simplicial ball, because $J \neq [d_1] \times \ldots \times [d_k]$.
\end{proof}

From the proof of the previous proposition, it is clear that in the case of simplicial complexes, i.e. order ideals in $[2]^k$, $M_d(J)$ is a completely balanced simplicial complex. 
\begin{cor}\label{corollario incredibbile}
The $f$-polynomial of a simplicial complex is the $h$-polynomial of a
completely balanced vertex-decomposable simplicial ball.
\end{cor}

Recall that a simplicial complex, i.e. an order ideal $J \subseteq [2]^k$, is {\em flag} if either $J=[2]^k$ or $\tilde{\rho}(x) \leq 2$ for all $x \in \min([2]^k \setminus J)$. We recover \cite[Proposition 4.1]{Costantinescu} thanks to Corollary \ref{corollario incredibbile} and the following proposition.
\begin{prop} \label{prop da medaglia fields}
The simplicial complex $M_d(J)$ is flag if and only if $J$ is a flag simplicial complex.    
\end{prop}

\begin{proof}
First of all we prove that if $J$ is a flag simplicial complex then $M_d(J)$ is flag. By our assumption we have that $J \subseteq [2]^k$ is an order ideal. Let $G:=\{(i_1,j_1),\ldots,(i_h,j_h)\}$ be a non-face of $M_d(J)$, where  $j_1 \leq \ldots \leq j_h$ and $i_t \in [2]$ for all $t \in [h]$ and $h \geq 3$. For every $j \in [k]$ the set $X_j:=\{(1,j),(2,j)\}$ is a non-face of $M_d(J)$ by construction. If $X_j \subseteq G$ for some $j \in [k]$, then $G$ is not minimal. So assume $X_j \not \subseteq G$ for all $j \in [k]$; this implies $j_1<\ldots<j_h$. Let $x \in [2]^k$ such that $x_{j_t}=i_t$ if $t \in [h]$ and $x_j=1$ for $j \notin \{j_1,\ldots,j_h\}$. Since $G \subseteq F_x$, then $F_x$ is not a facet of $M_d(J)$, i.e. $x \notin J$.  Since $J$ is flag, there exists $y \leq x$, $y \notin J$ such that the set $A:=\{i \in [k]:y_i=2\}$ satisfies $1 \leq |A| \leq 2$. It is clear that the set $\{(2,t): t\in A\}$ is contained in $G$ and it is not a face of $M_d(J)$, because otherwise there exists $z \in J$ such that $\{(2,t): t\in A\} \subseteq F_z$ and this implies $y \leq z$ and $y \in J$. It follows that every non-face of $M_d(J)$ of cardinality $\geq 3$ is not minimal, i.e. $M_d(J)$ is flag.

Let $\varnothing \neq J \subseteq [d_1] \times \ldots \times [d_k]$. We prove now that if $M_d(J)$ is flag then $J$ is a flag simplicial complex. Consider $m_i:=\max\{x_i:x \in J\}$. Notice that, for all $i \in [k]$, the set $\{(d_i+1-m_i,i),\ldots,(d_i,i)\}$ is a minimal non-face of $M_d(J)$ by construction. Hence, since $M_d(J)$ is flag, we have that $m_i \leq 2$, for all $i \in [k]$, i.e. $J$ is a simplicial complex. We need only to prove that $J$ is flag. If $\max(J)=\{(2,\ldots,2)\}$ we have nothing to prove. Then assume that $\max(J) \neq \{(2,\ldots,2)\}$. Let $x:=(x_1,\ldots,x_k)$ be a minimal non-face of $J$ and let $h:=|\{i \in [k]: x_i=2\}|$. Notice that $G:=\{(d_{i_1},i_1),\ldots,(d_{i_h},i_h)\}$ is a non-face of $M_d(J)$, where $\{i \in [k]: x_i=2\}=\{i_1,\ldots,i_h\}$ and $i_1<\ldots<i_h$. Indeed, if it is a face, there exists $z \in J$ such that $G \subseteq F_z$ and then $z \geq x$, but $x \notin J$. It is also clear that $G$ is minimal, by the minimality of $x$. It follows that $h \leq 2$ and this concludes the proof.
\end{proof}
By Proposition \ref{prop da medaglia fields} we obtain the following equality, which supports \cite[Conjecture 1.4]{Costantinescu}:
$$\{\mbox{$h$-vectors of flag $M$-complexes}\}=\{\mbox{$f$-vectors of flag simplicial complexes}\}.$$ 

\section{Lehmer complexes of lower Bruhat intervals} \label{sezione codici}
In this section we introduce Lehmer codes for finite Coxeter groups and we study the shape of their lower Bruhat intervals whenever a Lehmer code is available. 

By taking advantage of Definition \ref{def complesso}, for a finite Coxeter system $(W,S)$ we introduce a shellable simplicial complex whose $h$-polynomial coincides with the rank-generating function of the group $W$ with the Bruhat order.

\begin{dfn} \label{def complesso2} The {\em Lehmer complex} $L(W,S)$ of a finite Coxeter system $(W,S)$ with exponents $e_1,\ldots,e_n$ is the simplicial complex $M_d$, where $d:=(e_1+1,\ldots,e_n+1)$.
\end{dfn}

It is well-known that the rank-generating function of the poset $(W,\leq)$ is $\prod_{i=1}^n [e_i+1]_q$. Then, by Corollary \ref{corollario ideali}, it equals the $h$-polynomial of the Lehmer complex of $(W,S)$. 

It is worth to mention that Lehmer complexes are the analogues of Coxeter complexes with respect to the length statistic: both are shellable complexes and all the linear extensions of their underlying poset (the weak order in the case of the Coxeter complex) provide shelling orders, see \cite{bjorner2}.

The Lehmer codes of permutations provide a bijection between $S_n$ and $[1] \times \ldots \times [n]$. Such a bijective function is a poset morphism between $S_n$ with the left weak order and $[1] \times \ldots \times [n]$ ordered componentwise, see \cite[Lemma 2.4]{Denoncourt}. Moreover, it holds also that the inverse of this function is a poset morphism, whenever $S_n$ is endowed with the Bruhat order. This is stated in Corollary \ref{corollario codice}. In this section we extend the notion of Lehmer code to finite Coxeter groups. We recall that the relation $\leq$ on $W$ is the Bruhat order. 

\begin{dfn} \label{def codici di L}
    Let $(W,S)$ be a finite Coxeter system with exponents $e_1,\ldots,e_n$. A bijective function $L: W \rightarrow \prod_{i=1}^n\{0,1,\ldots,e_i\}$ is a \emph{Lehmer code} if
    $$L^{-1}: \prod\limits_{i=1}^n\{0,1,\ldots,e_i\} \rightarrow (W,\leq)$$
    is order preserving.
\end{dfn} If $L$ is a Lehmer code for $(W,S)$ and $n:=|S|$, then necessarily $$\ell(w)=\rho(L(w))=\sum_{i=1}^n L(w)_i,$$ where $\rho$ is the rank function of the poset $(\N\cup \{0\})^n$.  
For $w \in W$ we define

$$h_w(q):=\sum_{v\leq w}q^{\ell(v)},$$ the rank-generating function of the interval $[e,w]$, and 
$$J_w:=\{L(v)+(1,\ldots,1): v \leq w\}.$$

For all finite Coxeter groups with a Lehmer code we recover in a purely combinatorial way a result due to Bj\"orner and Ekedahl, see \cite[Theorem E]{bjorner}. They prove it for finite Weyl groups as a consequence of the geometry of their Schubert varieties. 

\begin{prop} \label{cor biorner}
   Let $L$ be a Lehmer code for a finite Coxeter system $(W,S)$. Then 
   the coefficients of $h_w(q)$ form an $M$-sequence, for all $w\in W$.
\end{prop}
\begin{proof} Let $e_1,\ldots,e_n$ be the exponents of $(W,S)$.
Since the inverse of a Lehmer code is a poset morphism,
the set $J_w$ is an order ideal of $[e_1+1]\times \ldots \times [e_n+1]$, i.e. a multicomplex. Clearly $h_w(q)$ equals the $f$-polynomial of $J_w$ and then its coefficients form an $M$-sequence.
\end{proof}

In view of the next section, the previous result applies to Coxeter systems of types $A_n$, $B_n$, $D_n$ and $H_3$.

Now we introduce a simplicial complex whose $h$-polynomial equals the rank-generating function of the lower Bruhat interval $[e,w]$.

\begin{dfn}\label{def per intervalli} Let $L$ be a Lehmer code for a finite Coxeter system $(W,S)$ with exponents $e_1,\ldots,e_n$; for $w \in W$ the {\em Lehmer complex $L_w(W,S)$ of $[e,w]$} is the simplicial complex $M_d(J_w)$, where $d:=(e_1+1,\ldots,e_n+1)$.
\end{dfn}

In view of the analogy between the complex $L(W,S)$ and the Coxeter complex of $(W,S)$, the shellable complex $L_w(W,S)$ is an analogue of a Coxeter cone as introduced in \cite{stembridge}. The next result is an immediate consequence of all the machinery developed up to this point.

\begin{thm} \label{teorema h pol ranghi}
    Let $L$ be a Lehmer code for a finite Coxeter system $(W,S)$ and $w \in W$. Then the $h$-polynomial of the simplicial complex $L_w(W,S)$ is equal to $h_w(q)$.
\end{thm}
\begin{proof}
    Since $\ell(w)=\sum_{i=1}^n L(w)_i=\rho(L(w))$, for all $w\in W$, by Corollary \ref{corollario ideali} the result follows.
\end{proof}

\section{Lehmer codes for some finite Coxeter groups} \label{sezione codici espliciti}
In this section we construct several Lehmer codes for some finite irreducible Coxeter systems. In general the reducible case can be traced back to the irreducible one in view of the following observation. If $(W,S) \simeq (W_1,S_1)\times (W_2,S_2)$, and $L_1$, $L_2$ are Lehmer codes for $(W_1,S_1)$ and $(W_2,S_2)$, respectively, then the assignment $$
w \mapsto L(w)=(L_1(w_1)_1,\ldots,L_1(w_1)_{|S_1|},L_2(w_2)_1,\ldots,L_2(w_2)_{|S_2|})$$ defines a Lehmer code for $(W,S)$, where $w=w_1w_2 \in W$, $w_1\in W_1$ and $w_2\in W_2$.

While a rank-preserving bijection $W \simeq \prod_{i=1}^n\{0,1,\ldots,e_i\}$ exists for all finite Coxeter systems, where $e_1,\ldots,e_n$ are the exponents of $(W,S)$, the additional condition to be a poset morphism is not fulfilled by all such bijections. 

\begin{oss}
The existence of Lehmer codes for classical Weyl groups, although in different terms, was already foreseen by R. Stanley in \cite[Section 7]{Stanley}. 
In \cite{billei}, S. Billey has proved that, in type $A_n$ and $B_n$, any lower Bruhat interval with palindromic Poincaré polynomial admits an immersed product of chains, i.e. it has a ``Lehmer code'', see \cite[Corollary 1.3]{billei}. We checked that this property holds also in type $H_3$.
\end{oss}

Given a Lehmer code for $(W,S)$, one can construct directly several other Lehmer codes. Indeed, it is clear that if $L$ is a Lehmer code for $(W,S)$ and $f$ is an automorphism of the poset $(W,\leq)$, then $L \circ f$ is a Lehmer code for $(W,S)$. In the dihedral system $I_2(m)$, the automorphism group of the Bruhat order is isomorphic to $\Z_2^{m-1}$. In all other cases the automorphism group of $(W,\leq)$ can be described in terms of the automorphism group of the Coxeter graph of $(W,S)$ and the inversion automorphism $w\mapsto w^{-1}$, see \cite[Section 2.3]{bjorner} and references therein. Denote by $L^*$ the \emph{dual of the Lehmer code $L$}, which we define by setting $$L^*(w):=L(w^{-1}),$$ for all $w\in W$. This is a Lehmer code by the previous discussion. 

One of the main tools for our constructions is the machinery described in \cite[Section 7.1]{BB} to factorize the Poincaré polynomial of $(W,S)$. We establish here some useful notation. Let $S=\{s_1,\ldots,s_n\}$ and $w\in W$. As proved in \cite[Corollary 2.4.6]{BB} (see also \cite[Proposition 3.4.2]{BB})
we have a factorization $$w=w^{(1)}w^{(2)}\cdots w^{(n)},$$ such that $w^{(i)}\in {^{\{s_1,\ldots,s_{i-1}\}}W_{\{s_1,\ldots,s_i\}}}$, for all $i\in [n]$, and $\ell(w)=\ell(w^{(1)})+\ldots + \ell(w^{(n)})$. Then the Poincaré polynomial of $(W,S)$ factorizes as $$W(q)=\prod_{i=1}^n {^{\{s_1,\ldots,s_{i-1}\}}W_{\{s_1,\ldots,s_i\}}}(q).$$ 

For the notation on irreducible Coxeter systems we follow \cite[Appendix A1]{BB}, sometimes with a different linear order on the generators.  We also observe that for every $u,v \in W$, if $u^{(i)} \leq v^{(i)}$ for all $i\in [n]$ then, by the subword property of the Bruhat order, $u=u^{(1)}\cdots u^{(n)}\leq v^{(1)}\cdots v^{(n)}=v$. This is crucial for the construction of some Lehmer codes. In the following subsections, for sake of clarity, we adopt the same notation for the type of a Coxeter system and the group defining it.

\subsection{Lehmer codes for type $I_2(m)$}
Let $m \in \{3,4,\ldots\}$.
The Coxeter system $(W,S)=I_2(m)$ has two generators $s,t$ which satisfy the relation $(st)^m=e$. The Poincaré polynomial factorizes as $W(q)=W_{\{s\}}(q)\cdot {^{\{s\}}W}(q)=[2]_q[m]_q$.
Let $$L_{I_2(m)}: W  \rightarrow \{0,1\}\times \{0,\ldots,m-1\}$$ be defined by setting $L_{I_2(m)}(w)=(\ell(w^{(1)}),\ell(w^{(2)}))$, for all $w\in W$.
Then the function $L_{I_2(m)}$ is a Lehmer code for $I_2(m)$. 
It is straightforward to check that there are $2^{m-1}$ distinct Lehmer codes for $I_2(m)$ and they are all obtained by composing $L_{I_2(m)}$ with an automorphism of $(W,\leq)$.

\subsection{A Lehmer code for type $A_n$} \label{sezione tipo a}
The Poincaré polynomial for Coxeter systems $(A_n,S)$ of type $A_n$ factorizes as $A_n(q)=[2]_q[3]_q\cdots [n+1]_q$. The quotients ${^{\{s_1,\ldots,s_{i-1}\}}(A_n)_{\{s_1,\ldots,s_i\}}}$ with the Bruhat order are chains, for all $i\in [n]$. Let $$L_{A_n}:A_n\rightarrow \prod\limits_{i=1}^n\{0,\ldots,i\},$$ be defined by $L_{A_n}(w)=(\ell(w^{(1)}),\ldots,\ell(w^{(n)}))$, for all $w\in A_n$. It is clear that
the function $L_{A_n}$ is a Lehmer code for type $A_n$. For example, if $w=s_2s_1s_3s_2$ then $L_{A_3}(w)=(0,2,2)$; indeed $w^{(1)}=e$, $w^{(2)}=s_2s_1$ and $w^{(3)}=s_3s_2$. 

\vspace{1em}
We now describe the Lehmer code $L_{A_{n-1}}$ in terms of inversions. Consider the symmetric group $S_n$ as a Coxeter group of type $A_{n-1}$. 
Let 
\begin{equation} \label{equazione Ln}
L_n(w):=(0,L_{A_{n-1}}(w)_1,\ldots,L_{A_{n-1}}(w)_{n-1}),\end{equation} for all $w\in S_n$. 

\begin{thm} \label{teorema cocodice} Let $n \geq 2$. Then 
    $$L_n(w)_k= |\{i \in [n]: i<k, \, w^{-1}(i)>w^{-1}(k)\}|,$$ for all $k\in [n]$, $w \in S_n$.
\end{thm}
\begin{proof}
    We prove the result by induction on $n$. If $n=2$ the result is trivial. Assume $n>2$. If $w(n)=n$, we have that $L_n(w)_n=0=|\{i \in [n-1]: w^{-1}(i)>w^{-1}(n)\}|$ and, by our inductive hypothesis,  $L_n(w)_k=|\{i \in [k-1]: w^{-1}(i)>w^{-1}(k)\}|$, for all $k<n$. 
    
    Let $a:=w^{-1}(n)< n$,  and $\widetilde{w}:=w^{(1)}\cdots w^{(n-2)}$. Then, for $k<n$, by the inductive hypothesis we have that
    \begin{eqnarray*}
       L_n(w)_k&=& L_{n-1}(\widetilde{w})_k\\
        &=& |\{i \in [k-1]: \widetilde{w}^{-1}(i)>\widetilde{w}^{-1}(k)\}| \\
         &=& |\{i \in [k-1]: w^{-1}(i)>w^{-1}(k)\}|,
    \end{eqnarray*}
    since $w^{(n-1)}=s_{n-1}s_{n-2}\cdots s_a$ and then $$w = \widetilde{w}(1)\ldots \widetilde{w}(a-1) \, n \, \widetilde{w}(a) \ldots \widetilde{w}(n-1).$$
    Moreover we have that \begin{eqnarray*}
        L_n(w)_n&=& \ell(w)-\ell(\widetilde{w})\\
        &=& |\{i \in [n-1]: w^{-1}(i)>w^{-1}(n)\}|
    \end{eqnarray*} and this concludes the proof.
\end{proof}

For $n\geq 3$, we denote by $\iota_n$ the non-trivial automorphism of the Coxeter graph of $A_{n-1}$. If $n=2$, $\iota_n$ is the identity. If $w_0$ is the element of maximal length in $S_n$, then $\iota_n(w)=w_0ww_0$, for all $w\in S_n$.
In the following corollary we recover what is known as \emph{Lehmer code of a permutation}, see e.g. \cite[Section 11.3]{Kerber} and \cite[Chapter 2]{Manivel}. It is worth to mention that the Lehmer code of a permutation provides the leading term of the corresponding Schubert polynomial, see \cite{Fulton}.
\begin{cor} \label{corollario codice} Let $n\geq 2$.
    Then $$|\{j\in [n]: j>i, \, w(j)<w(i)\}|=(L_n \circ \iota_n)^*(w)_{n-i+1},$$
    for all $i\in [n]$, $w\in S_n$.
\end{cor}
\begin{proof}
     By Theorem \ref{teorema cocodice} we obtain 
    \begin{eqnarray*}
        L_n(w_0w^{-1}w_0)_{n-i+1} &=& |\{j \in [n-i]: (w_0ww_0)(j)>(w_0ww_0)(n-i+1)\}| \\
        &=& |\{j \in [n-i]: w(i)>w(n-j+1)\}| \\
        &=& |\{j \in [n]: j>i, \, w(i)>w(j)\}|. 
    \end{eqnarray*} We have that
    $(L_n \circ \iota_n)^*(w)=(L_n \circ \iota_n)(w^{-1})=L_n(w_0w^{-1}w_0)$ and the result follows.
\end{proof}

By Corollary \ref{corollario codice}, what is known as the Lehmer code satisfies the condition of Definition \ref{def codici di L} and then it is a Lehmer code in our sense.

\subsection{A Lehmer code for type $B_n$}

\begin{figure}  
\begin{center}\begin{tikzcd}
s_1 \arrow[r, dash]{r}{4} & s_2 \arrow[r, dash] & s_3 \arrow[r, dash] & \cdots \arrow[r, dash] & s_n 
\end{tikzcd}\caption{Coxeter graph of $B_n$.}  \label{tipo B} \end{center} 
\end{figure}

We can construct a Lehmer code for type $B_n$ similarly to type $A_n$, as we state in the next theorem. Let $(B_n,S)$ be a Coxeter system of type $B_n$;
we take the linear order on $S$ as in Figure \ref{tipo B}.
\begin{thm} \label{teorema Bn}
    Let $n\geq 2$. Then the function $$L_{B_n}: B_n \rightarrow \prod\limits_{i=1}^n\{0,\ldots, 2i-1\}$$ defined by setting $L_{B_n}(w)=(\ell(w^{(1)}),\ldots,\ell(w^{(n)}))$, for all $w\in B_n$, is a Lehmer code.
\end{thm}
\begin{proof} If $n\geq 2$,
  we have that the maximal quotient ${^{S\setminus \{s_n\}}B_n}$ is a chain of rank $2n-1$. Indeed, $$\ell({^{S\setminus \{s_n\}}w_0})=\ell(w_0)-\ell(w_0(S\setminus \{s_n\}))=n^2-(n-1)^2=2n-1$$ and $$|{^{S\setminus \{s_n\}}B_n}|=\frac{2^nn!}{2^{n-1}(n-1)!}=2n.$$ This implies our statement.
\end{proof}

\begin{oss}
    It is worth to mention that Lehmer codes $L_{A_n}$ and $L_{B_n}$ have been considered in \cite{geck} for the study of bases for the Bruhat order in types $A_n$ and $B_n$. Moreover, in type $B_n$, a {\em signed Lehmer code} is defined in \cite{Smirnov}. 
\end{oss}  

\subsection{A Lehmer code for type $D_n$}

\vspace{1em}
Let $n\geq 4$ and $(D_n,S)$ be a Coxeter system of type $D_n$, whose generators are ordered as in Figure \ref{fig Coxeter Dn}.
\begin{figure}  
\begin{center}\begin{tikzcd}
& s_0 \arrow[d,dash]& & & \\
s_1 \arrow[r, dash] & s_2  \arrow[r, dash] & s_3 \arrow[r, dash] & \cdots \arrow[r, dash] & s_{n-1} 
\end{tikzcd}\caption{Coxeter graph of $D_n$.}  \label{fig Coxeter Dn} \end{center} 
\end{figure} We define $n$ chains in the Bruhat order of the group $D_n$ as follows: 
\begin{eqnarray*} 
    X_1 &:=& (D_n)_{\{s_1\}}\\
    X_2 &:=& {^{\{s_1\}}}(D_n)_{\{s_1,s_2\}} \cup \, s_0s_2s_1\{e\}\\
    X_3 &:=&  {^{\{s_1,s_2\}}}(D_n)_{\{s_1,s_2,s_3\}} \cup \, s_3s_0s_2s_1\{e,s_3\}\\
    &\vdots&  \\
     X_{n-1} &:=& {^{\{s_1,\ldots,s_{n-2}\}}}(D_n)_{\{s_1,\ldots,s_{n-1}\}}  \\
     && \cup \, \, s_{n-1}s_{n-2}\cdots s_3s_0s_2s_1\{e,s_3,\ldots, s_3s_4\cdots s_{n-1} \} \\
    Y_n &:=& \{e,s_0,s_0s_2,s_0s_2s_3,\ldots, s_0s_2s_3\cdots s_{n-1}\} 
   \end{eqnarray*}

We need the following result to define a Lehmer code for type $D_n$.

\begin{prop} \label{prop quoziente} Let $n\geq 4$. Then \,
$Y_{n-1} \cdot {^{S\setminus \{s_{n-1}\}}D_n} = X_{n-1}Y_n $.
\end{prop}
\begin{proof} We have that $|X_{n-1}|=2(n-1)$, $|Y_n|=n$ and $|^{S\setminus \{s_{n-1}\}}D_n|=\frac{2^{n-1}n!}{2^{n-2}(n-1)!}=2n$. Hence
\begin{eqnarray*}
|Y_{n-1} \cdot {^{S\setminus \{s_{n-1}\}}D_n}| &=& |Y_{n-1}|\cdot |{^{S\setminus \{s_{n-1}\}}D_n}| \\
&=& (n-1)\cdot 2n = |X_{n-1}|\cdot|Y_n|.
\end{eqnarray*}
Therefore, to prove our statement, it is sufficient to prove the inclusion 
$Y_{n-1} \cdot {^{S\setminus \{s_{n-1}\}}D_n} \, \subseteq \, X_{n-1}Y_n$.
First notice that

\begin{eqnarray*}
    {^{S\setminus \{s_{n-1}\}}D_n} &=& ^{S\setminus \{s_0,s_{n-1}\}}(D_n)_{\{s_1,\ldots,s_{n-1}\}} \\ && \uplus \, \, s_{n-1}\cdots s_2\{s_0,s_1s_0, s_1s_0s_2,\ldots,s_1s_0s_2\cdots s_{n-1}\} \\
    && \\
   &=& ^{S\setminus \{s_0,s_{n-1}\}}(D_n)_{\{s_1,\ldots,s_{n-1}\}} \uplus \, \, s_{n-1}\cdots s_2 (\{s_0\} \cup s_1(Y_n\setminus \{e\})).
\end{eqnarray*} 

Let $y\in Y_{n-1}$ and $u\in {^{S\setminus \{s_{n-1}\}}D_n}$. The previous equality implies that $u\in X_{n-1}Y_n$; so, if $y=e$, we have that $yu\in X_{n-1}Y_n$. Let $y=s_0$. If $u=e$ then $yu\in Y_{n-1}\subseteq Y_n \subseteq X_{n-1}Y_n$. If $u>e$ there are some cases to be considered:
\begin{enumerate}
    \item $u=s_{n-1}\cdots s_i$, $2<i \leq n-1$: in this case $yu=uy\in X_{n-1}Y_n$.
    \item $u=s_{n-1}\cdots s_2$: in this case $yu=s_{n-1}\cdots s_3 \cdot s_0s_2 \in X_{n-1}Y_n$.

    \item $u=s_{n-1}\cdots s_2s_1$: in this case $yu=s_{n-1}\cdots s_4 \cdot s_3s_0s_2s_1 \in X_{n-1}$.
    \item $u=s_{n-1}\cdots s_2s_0$: we have that $yu=s_{n-1}\cdots s_3s_2 \cdot s_0s_2 \in X_{n-1}Y_n$.
     \item $u=s_{n-1}\cdots s_2s_1 \cdot s_0$: we have that $yu=s_{n-1}\cdots s_3 s_0s_2s_1\cdot s_0 \in X_{n-1}Y_n$.
    \item $u=s_{n-1}\cdots s_2s_1 \cdot s_0s_2 \cdots s_i$, $2\leq i \leq n-1$: we have that $yu=s_{n-1}\cdots s_3s_0s_2s_1 \cdot s_0s_2 \cdots s_i \in X_{n-1}Y_n$.
\end{enumerate}
Now assume $\ell(y)>1$ and $u\in {^{\{s_1,\ldots,s_{n-2}\}}}(D_n)_{\{s_1,\ldots,s_{n-1}\}}$. Then $y=s_0s_2\cdots s_j$ for some $2\leq j \leq n-2$, and $u=s_{n-1}\cdots s_i$, for some $i \in [n-1]$, or $u=e$. If $u=e$ we have that $yu=y\in Y_n\subseteq X_{n-1}Y_n$. If $u\neq e$ we consider the following cases.

\begin{enumerate}
    \item $2 \leq j<i-1$: in this case $i>3$ and $yu=uy \in X_{n-1}Y_n$.
    \item $j=i-1$: we have that $i\geq 3$ and $yu=s_{n-1}\cdots s_{i+1} \cdot s_0s_2\cdots s_i\in X_{n-1}Y_n$.
   
\item $i-1<j< n-1$, $i\geq 2$: we have that $2 \leq i \leq j$ and
    \begin{eqnarray*}
        yu&=& s_0s_2\cdots s_j \cdot s_{n-1}\cdots s_{j+1}s_j\cdots s_i \\
        &=& s_{n-1}\cdots s_{j+2} \cdot s_0s_2 \cdots s_{i-1}(s_i \cdots s_j s_{j+1}s_j \cdots s_i) \\
        &=& s_{n-1}\cdots s_{j+2} \cdot s_0s_2 \cdots s_{i-1} (s_{j+1} \cdots s_{i+1} s_i s_{i+1} \cdots s_{j+1}) \\
        &=& s_{n-1}\cdots s_{i+1} \cdot s_0s_2 \cdots s_{i-1}s_i s_{i+1} \cdots s_{j+1} \in X_{n-1}Y_n.
    \end{eqnarray*}

    \item $2 \leq j < n-1$, $i= 1$: we have that 
    \begin{eqnarray*}
        yu&=&s_0s_2\cdots s_j\cdot s_{n-1}\cdots s_1 \\
        &=& s_{n-1}\cdots s_{j+2} s_0(s_2\cdots s_j s_{j+1} s_j\cdots s_2) s_1 \\
        &=& s_{n-1}\cdots s_{j+2} s_0(s_{j+1}\cdots s_3 s_2 s_3\cdots s_{j+1}) s_1 \\
        &=& s_{n-1}\cdots s_3 s_0s_2s_1 s_3\cdots s_{j+1} \in X_{n-1}.
    \end{eqnarray*}
\end{enumerate}
Now let $u=s_{n-1}\cdots s_2z$, with $z\in s_0\{e,s_1, s_1s_2,\ldots,s_1s_2\cdots s_{n-1}\}$. 
Then, by point $3.$ above, we have that:
\begin{enumerate}
    \item if $z=s_0$ then 
    \begin{eqnarray*}
        yu &=& s_{n-1}\cdots s_3 \cdot s_0s_2s_3 \cdots s_{j+1}s_0 \\
        &=& s_{n-1}\cdots s_3s_2 \cdot s_0s_2s_3 \cdots s_{j+1}  \in X_{n-1}Y_n;
    \end{eqnarray*}
    \item if $z=s_0s_1s_2\cdots s_i$, $1\leq i \leq n-1$ then
    \begin{eqnarray*}
        yu &=& s_{n-1}\cdots s_3 \cdot s_0s_2s_3 \cdots s_{j+1}s_0s_1s_2\cdots s_i \\
        &=& s_{n-1}\cdots s_3s_0s_2s_1s_3 \cdots s_{j+1} \cdot s_0s_2\cdots s_i \in X_{n-1}Y_n. \end{eqnarray*}
\end{enumerate}
\end{proof}
From the result of the previous proposition we deduce a factorization of the group $D_n$ realized by $n$ saturated chains of its Bruhat order.
\begin{cor} \label{corollario D}
    Let $n\geq 4$. Then $D_n=X_1X_2\cdots X_{n-1}Y_n$.
\end{cor}
\begin{proof} We claim that $D_n \subseteq X_1\cdots X_{n-2}X_{n-1}Y_n$, for all $n\geq 4$. We prove the claim by induction on $n$. Let $n=4$. We have that $$D_4=(D_4)_{\{s_1\}} \cdot {^{\{s_1\}}(D_4)}_{\{s_1,s_2\}} \cdot {^{\{s_1,s_2\}}(D_4)}_{\{s_0,s_1,s_2\}} \cdot {^{S\setminus \{s_3\}}D_4}$$
and ${^{\{s_1,s_2\}}(D_4)}_{\{s_0,s_1,s_2\}}=Y_3 \cup \{s_0s_2s_1\}$.
One can check that $${^{\{s_1\}}}(D_4)_{\{s_1,s_2\}} \cdot  s_0s_2s_1 =\{s_0s_2s_1,s_2s_0s_2s_1,s_2s_1s_0s_2s_1\} \subseteq X_2Y_3.$$
Therefore, by Proposition \ref{prop quoziente}, $D_4 \subseteq X_1X_2X_3Y_4$.
    Let $n>4$ and $w=w_J({^Jw}) \in D_n$, with $J:=S\setminus \{s_{n-1}\}$. Then $w_J \in D_{n-1}$ and,
    by our inductive hypothesis, we have that $w_J\in X_1\cdots X_{n-2}Y_{n-1}$. Hence, by Proposition \ref{prop quoziente}, we obtain $w_J({^Jw})\in X_1\cdots X_{n-2}X_{n-1}Y_n$ and this proves our claim. 
\end{proof}

The following theorem states one of the main results of the paper, i.e. the existence of a Lehmer code for type $D_n$, for all $n\geq 4$.
\begin{thm} Let $n\geq 4$. The function
    $$L_{D_n} : D_n \rightarrow \left(\prod\limits_{i=1}^{n-1}\{0,\ldots,2i-1\}\right) \times \{0,\ldots,n-1\}$$ defined by setting $$L_{D_n}(w)=(\ell(x_1),\ldots,\ell(x_{n-1}),\ell(y)),$$ for all $w=x_1 \ldots x_{n-1}y\in D_n=X_1 \ldots X_{n-1} Y_n$, is a Lehmer code for type $D_n$.
\end{thm}
\begin{proof}
    By Corollary  \ref{corollario D} we have $D_n=X_1X_2\cdots X_{n-1}Y_n$. Moreover $|D_n|=|X_1||X_2|\cdots |X_{n-1}||Y_n|$ and $X_1,\ldots,X_{n-1},Y_n$ are saturated chains of the Bruhat order; this implies $$\sum_{w\in D_n}q^{\ell(w)}=\sum_{x_1\in X_1}\ldots \sum_{x_{n-1}\in X_{n-1}}\sum_{y\in Y_n}q^{\ell(x_1)+\ell(x_2)+\ldots +\ell(x_{n-1})+\ell(y)}$$ 
    Assume that $\ell(w)=\ell(x_1x_2\ldots x_{n-1} y)< m:=\ell(x_1)+\ldots +\ell(x_{n-1})+\ell(y)$ with $m$ minimal with respect of this property. By the polynomial equality and the minimality of $m$, there exist $x_1',\ldots,x'_{n-1},y'$ such that $\ell(w)=\ell(x'_1)+\ell(x'_2)+\ldots +\ell(x'_{n-1})+\ell(y')<m$ and $w=x_1'\cdots x_{n-1}'y'$, a contradiction.
\end{proof}

\subsection{A Lehmer code for type $H_3$}\label{lehmer H3}

Let $(H_3,S)$ be a Coxeter system of type $H_3$. We order the set $S$ as in Figure \ref{H3}. A Lehmer code can be constructed in the following way.  
\begin{figure}  
\begin{center}\begin{tikzcd}
s_1 \arrow[r, dash] & s_2  \arrow[r, dash]{r}{5} & s_3
\end{tikzcd}\caption{Coxeter graph of $H_3$.}  \label{H3} \end{center} 
\end{figure} Define three chains $X$, $Y$ and $Z$ in the Bruhat order
of $H_3$ by setting
\begin{eqnarray*}
    X&:=&\{e,s_3, s_3s_2, s_3s_2s_3, s_3s_2s_3s_2, s_3s_2s_3s_2s_1,
s_3s_2s_3s_2s_1s_2,  \\
&& s_3s_2s_3s_2s_1s_2s_3,  s_3s_2s_3s_2s_1s_2s_3s_2, s_3s_2s_3s_2s_1s_2s_3s_2s_3\}, \\
Y&:=&\{e,s_2,s_2s_1,s_3s_2s_1, s_2s_3s_2s_1, s_1s_2s_3s_2s_1\},\\
Z&:=& \{s_1,s_1s_2,s_1s_2s_1,s_1s_3s_2s_1, s_2s_1s_3s_2s_1,s_1s_2s_1s_3s_2s_1\}.
\end{eqnarray*}
Notice that $X\subseteq {^{S\setminus \{s_3\}}H_3}$. In the following proposition we provide several characterizations of the set $Y\cup Z$. For a Coxeter system $(W,S)$ and a subset $V \subseteq W$, a generalized quotient is the set
$$W/V:=\{w \in W: \ell(wv)=\ell(w)+\ell(v), \, \forall \, v  \in V\}.$$
For definitions and results about generalized quotients in Coxeter groups we refer to \cite{bjorner3}. See also \cite[Exercise 2.22]{BB}. We denote by $[u,v]_L$ an interval in the left weak order of $W$.
\begin{prop} \label{prop da inventiones} Let $x_0:=\max X$ and $z_0:=\max Z$. Then
 $$ Y\uplus Z= [e,z_0]_L=(H_3)_{\{s_1,s_2\}}\{e,s_3s_2s_1\}=H_3/{\{x_0\}}=H_3/X.$$
\end{prop}
\begin{proof}
  Notice that $Y \cap Z = \varnothing$. It can be checked by hand that $Y\cup Z = [e,z_0]_L=(H_3)_{\{s_1,s_2\}}\{e,s_3s_2s_1\}$ and that $z_0x_0=w_0$. By \cite[Theorem 4.4]{bjorner3}, we have the equality $H_3/{\{x_0\}}=[e,z_0]_L$.
  Since $X$ is a saturated chain in the right weak order of $H_3$, we have that $H_3/X = H_3/ {\{x_0\}}$.
\end{proof}
The following corollary provides a factorization of a Coxeter group of type $H_3$ similar to a parabolic factorization $H_3^J \cdot (H_3)_J$. 
\begin{cor} \label{corollario fattorizzazione H3}
The Coxeter group $H_3$ factorizes as $H_3/X \cdot X$.    
\end{cor}
\begin{proof}  Notice that
    $\{e,s_3s_2s_1\}X \subseteq {^{S\setminus \{s_3\}}H_3}$.
    Moreover $|{^{S\setminus \{s_3\}}H_3}|=|H_3|/6=20$ and $|\{e,s_3s_2s_1\}X|=20$, so $\{e,s_3s_2s_1\}X = {^{S\setminus \{s_3\}}H_3}$. By Proposition \ref{prop da inventiones}
    we have the equality $Y\cup Z=(H_3)_{\{s_1,s_2\}}\{e,s_3s_2s_1\}$ and then $(Y\cup Z)X=H_3$.
\end{proof}
Now we are ready to define a Lehmer code in type $H_3$.  
\begin{thm} \label{lemer H3}
    The function $L_{H_3} : H_3 \rightarrow \{0,1\} \times \{0,\ldots,5\} \times \{0,\ldots,9\}$ defined by 
    $$
 L_{H_3}(ux) =  \left\{
  \begin{array}{ll}
    (0,\ell(u),\ell(x)), & \hbox{if $u\in Y$;} \\
    (1,\ell(u)-1,\ell(x)), & \hbox{if $u\in Z$,}
  \end{array}
\right.$$ for all $u\in H_3/X$, $x\in X$, is a Lehmer code.
\end{thm}
\begin{proof}
    By Corollary \ref{corollario fattorizzazione H3} the function $L_{H_3}$ is well defined and it is clearly bijective. Since $X$, $Y$ and $Z$ are saturated chains of the Bruhat order, to complete the proof we have only to check that $\ell(y)=\ell(z)-1$ implies $y < z$, for all $y\in Y$, $z\in Z$. This property is easily verified.
\end{proof}

In Example \ref{es palindromi H3} we show that the $L_{H_3}$-unimodal elements, defined in the next section, are in bijection with the palindromic Poincaré polynomials of the lower Bruhat intervals in $(H_3,\leq)$.

\begin{oss} \label{oss no riflessioni H3}
    We have checked, by using SageMath, that $(H_3,\leq_R)$ (equivalently $(H_3,\leq_L)$) does not inject, as poset, to $[2]\times [6] \times [10]$. This implies that there is no Lehmer code for $H_3$ defined in terms of reflections as can be done in types $A_n$. 
\end{oss}

\section{$L$-principal and $L$-unimodal elements} \label{sezione principale}
Given a Lehmer code $L$ of a finite Coxeter system $(W,S)$, we
introduce a partial order $\Lh$  on $W$ by setting 
$$u  \Lh v \, \, \Leftrightarrow \, \, L(u) \leq L(v),$$ for all $u,v \in W$. For $u \Lh v$, we let $[u,v]_{\mathrm{Lh}}:=\{z\in W: u\Lh z \Lh v\}$.
We have that $[u,v]_{\mathrm{Lh}}$ is a subposet of $[u,v]$. It is natural to study the elements $w \in W$ for which $[e,w]_{\mathrm{Lh}}=[e,w]$ as sets. For these elements, it is clear that the rank-generating function of $[e,w]$ factorizes as a product of $q$-analogs involving the entries of their Lehmer codes. Hence we define a distinguished set of elements of $W$, which we call \emph{$L$-principal}.

\begin{dfn} \label{def principali}
    Let $L$ be a Lehmer code for a finite Coxeter system $(W,S)$. An element $w$ is \emph{$L$-principal} if
    $[e,w]_{\mathrm{Lh}}=[e,w]$ as sets.
\end{dfn}

We denote by $\Pr(L)$ the set of $L$-principal elements in $W$. Notice that the poset $(W,\Lh)$ is a lattice. The lattice operations are the ones induced by the lattice $\imm(L)$, i.e. $$u \curlywedge v = L^{-1}(L(u) \wedge L(v)), \, \, \, u \curlyvee v = L^{-1}(L(u) \vee L(v)),$$ for all $u,v\in W$. It turns out that for any Lehmer code, $\Pr(L)$ is also a lattice. 

\begin{prop} \label{lemma reticolo}
Let $(W,S)$ be a finite Coxeter system with a Lehmer code $L$. Then the poset $\Pr(L)$ is a lattice.
\end{prop}
\begin{proof}
    Let $a_1,a_2\in \Pr(L)$, 
    $a:=a_1 \curlywedge a_2$ and $w\leq a$.
    We have that $a \Lh a_i$, for all $i\in [2]$. Since $a_i \in \Pr(L)$, we have that $a \leq a_i$, for all $i\in [2]$. Hence $w \leq a_i$, for all $i\in [2]$, i.e. $w \Lh a_i$, for all $i\in [2]$, by the $L$-principality of $a_i$. We have proved that $w\Lh a$ for all $w\leq a$, i.e. $a\in \Pr(L)$, and this concludes the proof.
\end{proof}
\begin{oss}\label{bruhat come segre} Notice that, by definition of $L$-principality, we have the equality of posets $(\Pr(L),\leq)=(\Pr(L),\Lh)$. In types $A_n$ and $B_n$, the principal elements for the Lehmer codes $L_{A_n}$ and $L_{B_n}$ defined in Section \ref{sezione codici espliciti} are known in literature as \emph{Kempf elements}.
\end{oss}

\begin{ex} By \cite[Propositions 4.3 and 4.9]{Huneke}, the elements of $\Pr(L_{A_n})$ and $\Pr(L_{B_n})$  are exactly the Kempf elements in type $A_n$ and $B_n$, respectively. It is known that Kempf elements in type $A_n$ coincide with $312$-avoiding permutations of $S_{n+1}$, see Section \ref{sezione6} and references therein. 
\end{ex}
 
It is clear that for an $L$-principal element $w\in W$ we have that the Poincaré polynomial of $[e,w]$ is $$h_w(q)=\prod\limits_{i=1}^{|S|}[L(w)_i+1]_q.$$ In general, for a finite Coxeter system, a more involved formula for the cardinality of $[e,w]$ holds, as we prove in the following theorem.

\begin{thm} \label{teorema maduro}
Let $L$ be a Lehmer code of $(W,S)$ and $w\in W$. Then 
     $$h_w(q)=\sum\limits_{\varnothing \neq X\subseteq \max\{L(v): v\leq w\}}(-1)^{|X|+1}\prod\limits_{i=1}^{|S|}[(\wedge_{x\in X}\,x)_i+1]_q.$$
\end{thm}
\begin{proof}
   For $w\in W$ we have that, as sets, $$[e,w]=\bigcup\limits_{z \in \max\{L(v): v\leq w\}}[e,L^{-1}(z)]_{\mathrm{Lh}}.$$ Moreover the rank-generating function of the poset $[e,x]_{\mathrm{Lh}} \cap [e,y]_{\mathrm{Lh}}=[e,x\curlywedge y]_{\mathrm{Lh}}$ is equal to the polynomial $\prod_{i=1}^{|S|}[L(x \curlywedge y)_i+1]_q$, for all $x,y\in W$. By the Inclusion-Exclusion principle we obtain the result. 
\end{proof}

\begin{ex}
    Let $w:=3412 \in S_4$ and $L:=L_{A_3}$. Then $L(w)=(0,2,2)$
    and $\max_{\Lh}[e,w]=\{2413,3214,3412\}$. We have that $2413 \curlywedge 3214 = 2134$, $2413 \curlywedge 3412 = 1423$, $3214 \curlywedge 3412 = 3124$ and $2413\curlywedge 3214 \curlywedge 3412 = 1234$. 
    Therefore
    \begin{eqnarray*}
        h_w(q)&=& \prod\limits_{i=1}^3[L(2413)_i+1]_q+\prod\limits_{i=1}^3[L(3214)_i+1]_q +\prod\limits_{i=1}^3[L(3412)_i+1]_q - \prod\limits_{i=1}^3[L(2134)_i+1]_q \\
        && -\prod\limits_{i=1}^3[L(1423)_i+1]_q
        -\prod\limits_{i=1}^3[L(3124)_i+1]_q
         + \prod\limits_{i=1}^3[L(1234)_i+1]_q \\
        &=& 2(1+q)(1+q+q^2)+(1+q+q^2)^2-(1+q)-2(1+q+q^2)+ 1
        \\ &=& 1+3q+5q^2+4q^3+q^4.
    \end{eqnarray*}
\end{ex} 
Now we define a subset of $\Pr(L)$ whose elements are involved in the description of palindromic Poincaré polynomials.
Given a Coxeter system $(W,S)$ with a Lehmer code $L$, we define the $L$-orbit of $w\in \Pr(L)$ by setting
$$O_w:=\{L(u): u \in \Pr(L), \, h_u(q)=h_w(q)\}.$$

\begin{dfn} \label{def unimodale}
We say that an $L$-principal element $w\in W$ is \emph{$L$-unimodal} if $L(w)=\min_{\leq_{\lex}} O_w$. 
\end{dfn}
We denote by $U(L)$ the set of $L$-unimodal elements in $W$. In Section \ref{sezione6}, for the Lehmer code $L_{A_n}$ of $S_{n+1}$, we prove that $U(L_{A_n})$ is the set of unimodal permutations in $S_{n+1}$.  

We observe that the inclusion $U(L)\subseteq \Pr(L)$ implies the poset equality $(U(L),\leq)=(U(L),\Lh)$, for every Lehmer code $L$.
Given a Coxeter system $(W,S)$ we define $$\Pal(W,S):=\left\{h_w(q): w\in W, \,  q^{\ell(w)}h_w(q^{-1})= h_w(q)\right\}.$$ 
It is clear by definition that 
\begin{equation}\label{equazione pal}\{h_w(q): w \in U(L)\} = \{h_w(q): w \in \Pr(L)\}\subseteq \Pal(W,S)\end{equation} 
and $|U(L)|=|\{h_w(q): w \in U(L)\}|$, for every Lehmer code $L$ of a finite Coxeter system $(W,S)$. In Section \ref{sezione6} we prove that in type $A_n$, for the Lehmer code $L_{A_n}$, we have the equalities
$$|\Pal(A_n)|=|\{h_w(q): w \in U(L_{A_n})\}|=2^n,$$ for all $n\in \N$.

In the following example we show that in general the previous equality does not hold. 

\begin{ex}
    It can be checked by using a computer that $|\Pal(B_3)|>|U(L_{B_3})|$. The linear order on the generators given in
    Figure \ref{tipo B 2} provides a Lehmer code $\tilde{L}_{B_n}$ for type $B_n$, by a construction similar to that in Theorem \ref{teorema Bn}. In this case we have that $|\Pal(B_3)|=|U(\tilde{L}_{B_3})|$ but $|\Pal(B_4)|>|U(\tilde{L}_{B_4})|$. Finally, similar computations show that $|\Pal(D_4)|=|U(L_{D_4})|$ and $|\Pal(D_5)|>|U(L_{D_5})|$.
    \begin{figure}  
\begin{center}\begin{tikzcd}
s_2 \arrow[r, dash]{r}{4} & s_1 \arrow[r, dash] & s_3 \arrow[r, dash] & \cdots \arrow[r, dash] & s_n 
\end{tikzcd}\caption{Coxeter graph of $B_n$.}  \label{tipo B 2} \end{center} 
\end{figure}
\end{ex}

In the next example we classify palindromic Poincaré polynomials of lower Bruhat intervals in type $H_3$. This classification relies on the set of unimodal elements with respect to the Lehmer code $L_{H_3}$ defined in Section \ref{lehmer H3}. 

\begin{ex} \label{es palindromi H3}
Consider the Lehmer code $L_{H_3}$ of Theorem \ref{lemer H3}. We have that 
   $$\Pal(H_3)= \left\{h_w(q): w \in \Pr(L_{H_3})\right\}=\left\{h_w(q): w \in U(L_{H_3})\right\},$$ where the first equality can be checked by a SageMath computation and the second equality follows from \eqref{equazione pal}.
Figure \ref{figura principale} depicts the poset of all palindromic Poincaré polynomials of lower Bruhat intervals in type $H_3$: each element $(x_1,x_2,x_3)$ in this poset corresponds to an $L_{H_3}$-unimodal element whose lower Bruhat interval has Poincaré polynomial equal to $[x_1+1]_q[x_2+1]_q[x_3+1]_q$.

\begin{figure} \begin{center}\begin{tikzpicture}
\matrix (a) [matrix of math nodes, column sep=0.9cm, row sep=0.7cm]{
    &    & (1,5,9)  &   &   \\
   &    & (1,5,4)  &   &   \\
     &    & (1,4,4)  &   &   \\
       &    & (1,3,4)  &   &   \\
         &    & (1,2,4)  &   &   \\
           & (1,1,4)   &   & (1,2,3)  &   \\
        (0,1,4)     &    & (1,1,3)  &   &  (1,2,2) \\
               &   (0,1,3) &   &  (1,1,2) &   \\
                 &  (0,1,2)  &   & (1,1,1)  &   \\
                   &    &  (0,1,1) &   &   \\
                     &    & (0,0,1)  &   &   \\
                       &    & (0,0,0)  &   &   \\};

\foreach \i/\j in {1-3/2-3, 2-3/3-3, 3-3/4-3, 4-3/5-3, 5-3/6-2, 5-3/6-4,
6-4/7-3, 6-4/7-5, 6-2/7-1, 6-2/7-3, 7-3/8-2, 7-3/8-4, 7-5/8-4, 7-1/8-2,
8-2/9-2, 8-4/9-2, 8-4/9-4, 9-2/10-3, 9-4/10-3, 10-3/11-3, 11-3/12-3}
    \draw (a-\i) -- (a-\j);
\end{tikzpicture} \caption{Hasse diagram of the Bruhat order of $U(L_{H_3})$.} \label{figura principale} \end{center} \end{figure}  
\end{ex}

\section{Unimodal permutations and Poincaré polynomials of smooth Schubert varieties} \label{sezione6} 

In this section we consider the symmetric group $S_n$ with its standard Coxeter presentation. For the Lehmer code $L_{A_{n-1}}$, principal elements are exactly Kempf elements in type $A_{n-1}$ and these are the permutations in $S_n$ avoiding the pattern $312$. We prove that $U(L_{A_{n-1}})$ is the set of unimodal permutations in $S_n$. The main result of this section is the classification of Poincaré polynomials of smooth Schubert varieties in complex flag manifolds. These are realized by the rank-generating functions for lower Bruhat intervals of unimodal permutations, as we prove in Theorem \ref{teorema lisce}.

Let $P(n)$ be the set of partitions of a non-negative integer $n$. To write a partition we adopt French notation. Then an element $\lambda \in P(n)$ is either $(\,)$, if $n=0$, or a tuple $(\lambda_1,\ldots,\lambda_r)$, where $0<\lambda_1\leq \lambda_2\leq \ldots \leq \lambda_r$ and $\sum_{i=1}^r\lambda_i=n$, if $n>0$. Given a partition $(\lambda_1,\ldots,\lambda_r)\in P(n)$, we define $[\lambda]:=[r]$ and, given $k \in \mathbb{N}$, $m_k(\lambda):=|\{i \in [r]: \lambda_i=k\}|$.
We denote by $\lambda^*$ the \emph{dual partition} (or conjugate partition), i.e. the unique partition determined by the condition $m_k(\lambda^*):=\lambda_{r-k+1}-\lambda_{r-k}$, for all $k\in [r]$, where $\lambda_0:=0$. The partition $(\,) \in P(0)$ is self-dual.

The set of \emph{smooth permutations} in $S_n$ is defined by setting $$\Sm_n:=\{w\in S_n: \mbox{$w$ avoids $3412$ and $4231$}\}.$$ 
It is known that $w\in \Sm_n$ implies 
$$h_w(q)= \sum\limits_{v\leq w}q^{\ell(v)}=\prod\limits_{i\in [E(w)]}[E(w)_i+1]_q,$$
where $E: \Sm_n \rightarrow \bigcup_{k\geq 0}P(k)$ is the function which assigns to any element $w\in \Sm_n$ its \emph{exponents} (see \cite[Theorem 1.1]{gasharov}). On other hand, if $h_w(q)$ is product of $q$-analogs then $w$ is smooth.
To define $E(w)$, following \cite[Remark 2.8]{gasharov}, we need some preparation.

For $w\in S_n$ define a poset $P_w =([n],\leq_w)$ by setting $i<_w j$ if and only if $i<j$ and $w^{-1}(j)< w^{-1}(i)$, for all $i,j\in [n]$. The comparability graph of the poset $P_w$ is known as a \emph{permutation graph}. For a finite poset $P$, we define the poset $C(P)$ of saturated chains of $P$, ordered by inclusion. It is a ranked poset; we denote by $\hat{\rho}$ its rank function and define $\hat{\rho}_i:=|\{c\in C(P):\hat{\rho}(c)=i\}|$, for all $i\geq 0$, and $m_P:=\max\{\hat{\rho}(c): c\in C(P)\}$. Proposition \ref{prop foreste} allows us to define a partition into distinct parts for any smooth permutation. Let $Q:=P_{3412}$ be the poset on $\{1,2,3,4\}$ with relations $\{(1,3),(1,4),(2,3),(2,4)\}$.

\begin{prop} \label{prop foreste}
    Let $P$ be a poset whose Hasse diagram is a forest. If $P$ avoids $Q$ as induced subposet then $\hat{\rho}_h>\hat{\rho}_{h+1}$, for all $0\leq h \leq m_P$.
\end{prop}

\begin{proof}
It is sufficient to prove the result for trees. If $m_P=0$ the result is trivial. The condition $m_P>0$ implies $|P|>1$.
    We proceed by induction on $|P|$. If $|P|=2$ then $m_P=1$ and $\hat{\rho}_0=2>1=\hat{\rho}_1$. Let $|P|>2$. We claim that there exists a leaf $x\in P$ belonging to a unique maximal saturated chain. 
    For $p\in P$ define $p^{\vartriangleleft}:=\{z\in P: p\vartriangleleft z\}$ and $p^{\vartriangleright}:=\{z\in P: p\vartriangleright z\}$. Notice that, since $P$ avoids $Q$, $\{p\in P: |p^{\vartriangleleft}|>1 \} \cap \{p\in P: |p^{\vartriangleright}|>1 \}=\varnothing$.
    Consider the induced subposet of $P$ on the set $X=\{p\in P: |p^{\vartriangleleft}|>1 \} \uplus \{p\in P: |p^{\vartriangleright}|>1 \}$. If $X=\varnothing$ then the claim is clearly true because in this case $P$ is a chain. Assume $X\neq \varnothing$. Then $X$ is a forest, being an induced subposet of a tree which avoids $Q$. Let $y\in X$ be a leaf of the forest $X$. Since $y$ is a leaf of the forest $X$, there is at most one element in $X$ that is adjacent to $y$ in $X$. Because $|y^{\vartriangleleft}| > 1$, $y$ has at least two covers in $P$, which define at least two disjoint upward branches. At most one of these branches can contain elements of $X$. Therefore, there exists a cover $w$ of $y$ such that the entire upward branch from $w$ contains no elements of $X$. Let $c$ be a maximal saturated chain passing through $y$ and $w$. By our choice of $w$, the intersection $c \cap X$ contains $y$ and can contain only elements $z\in X$ such that $|z^{\vartriangleleft}| > 1$, because $P$ avoids $Q$. Thus we have $|v^{\vartriangleright}| \leq 1$, for all $v \in c$. Hence, the element $x = \max c$ belongs to a unique maximal saturated chain. Similarly we obtain the claim in the case $|y^{\vartriangleright}|>1$.
    
  Let $P':=P\setminus \{x\}$. Then $\hat{\rho}'_j\in \{\hat{\rho}_j-1, \hat{\rho}_j\}$, for all $0\leq j\leq  m_{P'}$. If $m_{P'}=m_P-1$ then $\hat{\rho}_{m_P-1}> 1=\hat{\rho}_{m_P}$ because there are at least two saturated chains of rank $m_P-1$ in $P$, one containing $x$ and one not containing $x$. Moreover  $\hat{\rho}_h=\hat{\rho}'_h+1> \hat{\rho}'_{h+1}+1 = \hat{\rho}_{h+1}$, for all $0\leq h\leq m_P-1$, by our inductive hypothesis.  Assume then $m_{P'}=m_P$.
    We have that
    $\hat{\rho}'_a=\hat{\rho}_a$ implies $\hat{\rho}'_b=\hat{\rho}_b$, for all $0\leq a<b \leq m_P$.
    Let $0 \leq h \leq m_P$. Hence, if $\hat{\rho}'_h=\hat{\rho}_h$, by our inductive hypothesis we have
    $\hat{\rho}_h=\hat{\rho}'_h> \hat{\rho}'_{h+1} =\hat{\rho}_{h+1}$.
    If $\hat{\rho}'_h=\hat{\rho}_h-1$, again by our inductive hypothesis,
    $\hat{\rho}_h=\hat{\rho}'_h+1> \hat{\rho}'_{h+1}+1 \geq \hat{\rho}_{h+1}$.
\end{proof} Given a smooth permutation $w$, the Hasse diagram of the poset $P_w$ is a forest (see \cite[Section 3]{Forest-like permutations}) which avoids the poset $P_{3412}$.  The following example shows that the statement of Proposition \ref{prop foreste} is more general: there exists a forest which avoids the poset $Q$, but it is not isomorphic to $P_w$ for any permutation $w$.

\begin{figure}[htbp]
\centering
\begin{tikzpicture}[scale=1.5]
    \node (b1) at (-1.5, 2) {$4$};
    \node (x)  at (0.5, 2) {$5$};
    \node (b2) at (1.5, 2) {$6$};
    \node (b3) at (2.5, 2) {$7$};

    \node (a1) at (-1.5, 0) {$1$};
    \node (a2) at (0.5, 0) {$2$};
    \node (a3) at (2.5, 0) {$3$};

    \draw[thick] (a1) -- (x);
    \draw[thick] (a2) -- (x);
    \draw[thick] (a3) -- (x);

    \draw[thick] (a1) -- (b1);
    \draw[thick] (a2) -- (b2);
    \draw[thick] (a3) -- (b3);
\end{tikzpicture}
\caption{A tree of order dimension $3$ avoiding $P_{3412}$.}
\label{fig:poset_dim3}
\end{figure}

\begin{ex}
Consider the poset whose Hasse diagram is depicted in Figure \ref{fig:poset_dim3}. It is a tree which avoids $Q$. One can check that the dimension of this poset is $3$. It is known that a poset is isomorphic to $P_w$ for some permutation $w$ if and only if it has dimension less than or equal to $2$ (see \cite[Exercise 3.20]{StaEC1} and references therein).
\end{ex}

For $w\in \Sm_n \setminus \{e\}$, by considering the poset $C(P_w)$  and Proposition \ref{prop foreste}, one can define a partition 
\begin{equation} \label{eq partizione}
  \lambda_w:=(\hat{\rho}_{m_{P_w}},\ldots, \hat{\rho}_1) \in P(\ell(w)). 
\end{equation} We let $\lambda_e:=(\,)$, the partition with $0$ parts.
Then the function $E$ is defined by setting $E(w):=\lambda^*_w$, for all $w\in \Sm_n$.

In \eqref{equazione Ln} we have defined a function $L_n : S_n \rightarrow \prod_{i=0}^{n-1}\{0,\ldots,i\}$. By Theorem \ref{teorema cocodice}, this function satisfies $$L_n(w)_k= |\{i \in [n]: i<k, \, w^{-1}(i)>w^{-1}(k)\}|,$$ for all $k\in [n]$, $w\in S_n$. With a slight abuse of language, from now on we say that $L_n$ is a Lehmer code for $S_n$.
In order to describe $L_n$-principal elements in Proposition \ref{teorema catalani}, we study permutations $w \in S_n$ avoiding the pattern $312$. For these permutations we can easily relate the partition $\lambda^*_w$ to the Lehmer code $L_n(w)$.  
\begin{prop} \label{prop 312}
    Let $w\in S_n$ be a $312$-avoiding permutation and $\lambda := \lambda^*_w$. Then
    $$\{\lambda_h: h \in [\lambda]\} \cup \{0\}= \{L_n(w)_k: k\in [n]\}$$ and
    $$h_w(q)=\prod\limits_{k=1}^n[L_n(w)_k+1]_q.$$
\end{prop}
\begin{proof}
 If $w\in S_n$ avoids $312$ then $w\in \Sm_n$. If $w=e$ the result is trivial. Let $w\neq e$. By definition of the poset $P_w$, it is clear that for each inversion $(i,j)$ of $w$ we have a saturated chain $c_{(i,j)}$. Let $k \in [n]$ and define $I_k(w):=\{(i,k): i<k, w^{-1}(i)>w^{-1}(k)\}$. For a $312$-avoiding permutation $w$, if $(i,k),(j,k) \in I_k(w)$, for some $i<j$, then in the poset of saturated chains $C(P_w)$ we have that $\hat{\rho}(c_{(i,k)})>\hat{\rho}(c_{(j,k)})$. This implies that $|I_k(w)|$ is a component of $\lambda_w^*$, for all $k \in [n]$ such that $I_k(w) \neq \varnothing$. Since $|I_k(w)|=L_n(w)_k$ for all $k\in [n]$, the first statement follows. The second one is immediate from the fact that $E(w)=\lambda^*_w$.
\end{proof}

We proceed now to define \emph{Fubini words} and \emph{lazy Fubini words}. Fubini words are enumerated by Fubini numbers (see \cite[Section 3]{Wilson}). Lazy Fubini words are enumerated by Catalan numbers and they arise as Lehmer codes of $L_n$-principal permutations, as stated in Proposition \ref{teorema catalani}. 
\begin{dfn}
    Let $k\in \N$. An element $x\in \{0,\ldots,k-1\}^k$ is a \emph{Fubini word} if $\{x_1,\ldots, x_k\}=\{0,\ldots, \max(x)\}$. 
\end{dfn} As ordered set partitions, Fubini words index bases of generalized coinvariant algebras, as treated in \cite{Haglung}.
Fubini words have a geometric interest as they realize indexing sets for stratifications of spaces of line configurations, see \cite{Rhoades}. 
Since the dual of a partition into distinct parts, when a zero prefix is added, is a Fubini word, we obtain the following corollary of Proposition \ref{prop foreste}. 
\begin{cor} \label{cor foreste}
Let $w\in \Sm_n$. Then the dual partition $\lambda_w^*$, with a $0$ prefix if $\lambda_w^*\neq (0)$, is a Fubini word.
\end{cor}
The next definition introduces the notion of lazy Fubini word.
\begin{dfn} \label{prop fubini}
    An element $x\in \{0,\ldots,k-1\}^k$ is a \emph{lazy Fubini word} if $x_1=0$ and $x_{i+1}-x_i \leq 1$, for all $i\in [k-1]$.
\end{dfn}

Notice that lazy Fubini words are Fubini words. We denote by $\hat{\mathrm{F}}_k$ the set of lazy Fubini words in $\{0,\ldots,k-1\}^k$.
We have that $\hat{\mathrm{F}}_k \subseteq \prod_{i=0}^{k-1}\{0,\ldots,i\}$. Moreover, by Definition \ref{prop fubini} and \cite[Exercise 6.19(u)]{StaEC2}, lazy Fubini words are enumerated by Catalan numbers, i.e. $|\hat{\mathrm{F}}_k|=C_k$, for all $k\in \N$. 

\begin{ex}
    We have that $\hat{\mathrm{F}}_3=\{(0,0,0),(0,1,0),(0,0,1),(0,1,1),(0,1,2)\}$ and 
   \begin{eqnarray*}
     \hat{\mathrm{F}}_{4} &=&   \{(0,0,0,0),(0,1,0,0),(0,0,1,0),(0,0,0,1),(0,1,1,0), \\
     && (0,1,0,1), (0,0,1,1), (0,1,1,1),(0,1,2,0),(0,0,1,2), \\
     && (0,1,1,2),(0,1,2,1),(0,1,2,2),(0,1,2,3)\}.
   \end{eqnarray*}
    The tuple $(0,2,1,1)$ is a Fubini word which is not lazy. 
\end{ex}
Recall from Section \ref{sezione principale} that an element $w\in S_n$ is $L$-principal for a Lehmer code $L$ of $S_n$ if $[e,w]_{\mathrm{Lh}}=[e,w]$ as sets, and that we denote by $\Pr(L)\subseteq S_n$ the set of $L$-principal elements. The content of the following proposition is essentially known in the context of Kempf varieties, see \cite{Huneke}. Since a unifying result is not available in our setting, we think it is useful to state it here with a short proof. 

\begin{prop} \label{teorema catalani} Let $n\in \N$ and $w\in S_n$. The following are equivalent:
\begin{enumerate}
    \item[$\mathrm{1)}$] $w$ is $L_n$-principal;
    \item[$\mathrm{2)}$] $L_n(w)$ is a lazy Fubini word;
    \item[$\mathrm{3)}$] $w$ avoids the pattern $312$.
\end{enumerate}
    In particular $|\Pr(L_n)|=C_n$, where $C_n$ is the $n$-th Catalan number. 
\end{prop}
\begin{proof}
$1) \Rightarrow 2)$: by contradiction, assume $L_n(w)$ is not lazy Fubini. Hence by Definition \ref{prop fubini} there exists $i\in [n-1]$ such that $L_n(w)_i < L_n(w)_{i+1}-1$. Then there exists $a<i$ such that $w^{-1}(i+1)<w^{-1}(a)<w^{-1}(i)$.
We have that $s_iw< w$ in the Bruhat order. But $L_n(s_iw)_i > L_n(w)_i$ and then $s_iw\not \in [e,w]_{\mathrm{Lh}}$, i.e. $[e,w]_{\mathrm{Lh}} \neq [e,w]$, a contradiction.

  $2) \Rightarrow 3)$: by contradiction, assume that $w$ has the pattern $312$. This means that there exist $b<c<a$ with $a,b,c \in [n]$ such that $w^{-1}(a)<w^{-1}(b)<w^{-1}(c)$. Consider $a,b,c$ such that $(w^{-1}(a),w^{-1}(b),w^{-1}(c))$ is lexicographically minimal. The minimality of $(w^{-1}(a),w^{-1}(b),w^{-1}(c))$ implies that  $w^{-1}(k)>w^{-1}(a)$ for all $c+1 \leq k \leq a-1$. Consider $$m:=\min \{k \in \N: c+1\leq k\leq a,\, w^{-1}(a)\leq w^{-1}(k)<w^{-1}(b)\}.$$ Hence $m>2$ and $w^{-1}(m-1)>w^{-1}(b)$. The inequality $[L_n(w)]_m-[L_n(w)]_{m-1}>1$ follows, again a contradiction.

  $3) \Rightarrow 1)$: by Proposition \ref{prop 312} we have that $$|[e,w]|=\prod\limits_{i=1}^n (L_n(w)_i+1)=|[e,w]_{\mathrm{Lh}}|.$$ Since $[e,w]_{\mathrm{Lh}}$ is a subposet of $[e,w]$, we obtain that, as sets, $[e,w]_{\mathrm{Lh}}=[e,w]$, i.e. $w$ is $L_n$-principal.

The last assertion of our proposition follows from the fact that $312$-avoiding permutations are enumerated by Catalan numbers (see \cite[Exercise 6.19(ff)]{StaEC2}).
\end{proof}
\begin{oss}
By Proposition \ref{lemma reticolo}, Remark \ref{bruhat come segre} and Proposition \ref{teorema catalani}, we recover the lattice property of the Bruhat order on the set of $312$-avoiding permutations, stated in \cite[Theorem 5.1]{Ferrari}. This is also known as \emph{Dyck lattice} or \emph{Stanley lattice}, see \cite[Corollary 5.1]{Ferrari} and \cite{Bernardi}. \end{oss}

In the following definition, we recall the notion of {\em unimodal permutation}. For more details and results involving this class of permutations see for instance \cite{AL}, \cite{Gannon}, \cite{Thibon}.

\begin{dfn}
    A permutation $w\in S_n$ is \emph{unimodal} if
    $w(1)<\ldots<w(i)>w(i+1)>\ldots >w(n)$, where $i=w^{-1}(n)$.
\end{dfn} We denote by $U_n$ the set of unimodal permutations in $S_n$.
Since unimodal permutations coincide with permutations avoiding $312$ and $213$, Proposition \ref{teorema catalani} implies that every unimodal permutation is $L_n$-principal. In the next proposition we provide a further characterization of unimodal permutations.

\begin{prop} \label{prop unimodale lehmer}
Let $n\in \N$ and $u\in S_n$. The following are equivalent:
\begin{enumerate}
    \item[$\mathrm{1)}$] $u$ is $L_n$-unimodal;
    \item[$\mathrm{2)}$] $L_n(u)$ is a weakly increasing Fubini word;
    \item[$\mathrm{3)}$] $u$ is a unimodal permutation.
\end{enumerate}
\end{prop}
\begin{proof}
    First we prove that $2)$ is equivalent to $3)$. Let $u \in U_n$. Since $u$ avoids $312$, by Proposition \ref{teorema catalani} we have that $L_n(u)$ is a Fubini word. 
    Moreover
     $$L_n(u)_{i+1}=\left \{
  \begin{array}{ll}
    L_n(u)_i, & \hbox{if $u^{-1}(i)<u^{-1}(i+1)$;} \\ 
    L_n(u)_{i}+1, & \hbox{if $u^{-1}(i)>u^{-1}(i+1)$,}
  \end{array}\right.$$ for all $i \in [n-1]$. Hence $L_n(u)$ is a weakly increasing Fubini word.
  
  Let $f:=(0,f_2,\ldots, f_n) \in \{0,\ldots,n-1\}^n$ be a weakly increasing Fubini word; this implies $f_i\leq f_{i+1}$ for all $i\in [n-1]$, and $f_i \leq i-1$, for all $i\in [n]$. We proceed by induction on $n \geq 1$ to prove that $u:=L_n^{-1}(f) \in U_n$. If $n=1$ there is nothing to prove. 
      Let $n>1$. Then $\tilde{u}:=L^{-1}_{n-1}(0,f_2,\ldots,f_{n-1})\in U_{n-1}$ by our inductive hypothesis, because $(0,f_2,\ldots,f_{n-1})$ is still a weakly increasing Fubini word. Since $f$ is a weakly increasing Fubini word, $f_n\in \{f_{n-1}, f_{n-1}+1\}$. If $f_n=f_{n-1}$ we have that $u^{-1}(n)=\tilde{u}^{-1}(n-1)+1$.
      If $f_n=f_{n-1}+1$ we have that $u^{-1}(n)=\tilde{u}^{-1}(n-1)$.
      In both cases $u\in U_n$.

We prove now the equivalence between $1)$ and $2)$. Let $u \in U(L_n)$. Since $u$ is $L_n$-principal, $L_n(u)$ is a lazy Fubini word by Proposition \ref{teorema catalani}. Since $L_n(u)=\min_{\leq_{\lex}} O_u$ it follows immediately that $L_n(u)$ is a weakly increasing Fubini word. Finally, let $u \in S_n$ such that $L_n(u)$ is a weakly increasing Fubini word; then it is a lazy Fubini word. By Proposition \ref{teorema catalani}, $u$ is $L_n$-principal and $L_n(u)$ is the lexicographically minimum in its orbit because $L_n(u)$ is weakly increasing.
\end{proof}

In the next result we prove that all the possible Poincaré polynomials of smooth permutations in $S_n$ arise as Poincaré polynomials of unimodal permutations. 
\begin{thm} \label{teorema lisce} Let $n\in \N$. Then
    $$\{h_w(q) : w \in \Sm_n\} = \{h_w(q) : w \in U_n\}.$$ In particular,
    $|\{h_w(q) : w \in \Sm_n\}|=2^{n-1}$.
\end{thm}
\begin{proof}
The inclusion $\{h_w(q) : w \in U_n\} \subseteq \{h_w(q) : w \in \Sm_n\}$ is trivial because every unimodal permutation avoids the patterns $3412$ and $4231$.

 Let $w\in \Sm_n$. Consider the dual partition $\lambda^*_w$ of the partition $\lambda_w$ defined in \eqref{eq partizione}. Define $\widetilde{\lambda}_w$ equal to $\lambda^*_w$ to which we added as a prefix $n-k$ zeroes, where $k$ is the number of parts of $\lambda^*_w$. By Corollary \ref{cor foreste}, $\widetilde{\lambda}_w$ is a weakly increasing Fubini word. 
By Proposition \ref{prop unimodale lehmer}, 
$u:=L^{-1}_n(\widetilde{\lambda}_w)$ is a unimodal permutation.  By \cite[Remark 2.8]{gasharov} and Proposition \ref{prop 312}, $h_w(q)=h_u(q)$.
      The last assertion is straightforward because $L_n$ is a bijection and $|U_n|=2^{n-1}$. \end{proof}

Unimodal permutations of length $k$ in $S_n$ are in bijection with partitions $(\lambda_1,\ldots,\lambda_r)$ of $k$ such that $\lambda_1<\ldots<\lambda_r \leq n-1$, see \cite[Theorem 10]{AL}. We denote by $\Lambda$ this bijection and then by $\Lambda(u)$ the partition associated to $u \in U_n$. We recall how this bijection works. We let $\Lambda(e):=(\,)$, the partition with $0$ parts. Given $u \in U_n \setminus \{e\}$, let $j:=u^{-1}(n)$. Then $$\Lambda(u)=(n-u(j+1),\ldots,n-u(n)).$$

Conversely, given a partition $\lambda=(\lambda_1,\ldots,\lambda_r)$ of $k$, such that $\lambda_1<\ldots<\lambda_r \leq n-1$, consider the set $\{t_1,\ldots,t_{n-r}\}:=[n] \setminus \{n-\lambda_1,\ldots,n-\lambda_r\}$ and assume that $t_1<\ldots<t_{n-r}$. Then $$\Lambda^{-1}(\lambda):=t_1 \ldots t_{n-r} \, \, n-\lambda_1 \ldots n-\lambda_r.$$ 

The last result of the paper allows us to obtain the partition in \eqref{eq partizione}, for unimodal permutations, in a considerably simplified way by using the bijection $\Lambda$. 

\begin{prop} \label{roba da annals}
    Let $u\in U_n$. Then $\Lambda(u)=\lambda_u$ and 
    $$(\underbrace{0,\ldots,0}\limits_{u(n)},\Lambda(u)^*_1,\ldots,\Lambda(u)^*_s)=L_n(u),$$
where $s=n-u(n)$ is the number of parts of $\Lambda(u)^*$.
\end{prop}
\begin{proof} Let $\Lambda(u)=(\lambda_1,\ldots,\lambda_r)$ and $\Lambda(u)^*=(\mu_1,\ldots,\mu_s)$.
    We prove that  $(0,\ldots,0,\mu_1,\ldots,\mu_s)$ is equal to $L_n(u)$. If $\ell(u)=0$ we have nothing to show. Let $\ell(u)>0$.
    It is clear that $L_n(u)_i=0$ for all $i\leq u(n)$ and $[L_n(u)]_{u(n)+1}\neq 0$. We have that $s=\lambda_r=n-u(n)$ and the number of zeros in $L_n(u)$ is $u(n)=n-s$.  

    Let $0<k\leq n-u^{-1}(n)=r$; since $u$ is unimodal, we have that $$|\{i\in [n]:L_n(u)_i=k\}|=u(n-k)-u(n-k+1).$$ Hence, by defining $\lambda_0:=0$, we obtain 
    \begin{eqnarray*}
        u(n-k)-u(n-k+1)&=&n-u(n-k+1)-(n-u(n-k)) \\
        &=& \lambda_{r-k+1}-\lambda_{r-k}
    \end{eqnarray*}  and this concludes the proof of the second equality by definition of dual partition. 

    By  \cite[Remark 2.8]{gasharov} we have that $\lambda_{u}^*$ provides the exponents of $u$. By Proposition \ref{prop 312} also the Lehmer code $L_n(u)$ provides the exponents of $u$. Hence, by the equality proved above,
    $\Lambda(u)^*=\lambda_{u}^*$ and then $\Lambda(u)=\lambda_{u}$.
\end{proof}

By using Theorem \ref{teorema lisce} and Proposition \ref{roba da annals}, one can easily list all the possible Poincaré polynomials of smooth permutations in $S_n$, as the following example shows.

\begin{ex} \label{esempio finale}
    The partitions corresponding to the elements of $U_4$ are
    $$(1,2,3), (2,3), (1,3), (1,2), (3),  (2), (1), (\,).$$
    Then the elements of the set $\{h_w(q):w\in \Sm_4\}$ are
$$[2]_q[3]_q[4]_q, \,\, [2]_q[3]_q[3]_q, \,\, [2]_q[2]_q[3]_q, \,\, [2]_q[3]_q, \,\, [2]_q[2]_q[2]_q, \,\, [2]_q[2]_q, \,\, [2]_q, \,\, [1]_q.$$
\end{ex}

\begin{oss}
    By \cite[Theorem 1.6]{Nigro} it follows that $\{h_w(q):w \in \Sm_n\}=\{h_w(q):w \in \Pr(L_n)\}$. Hence Theorem \ref{teorema lisce} can be seen as a refinement of this equality, which is a special case of a conjecture due to Haiman \cite[Conjecture 3.1]{Haiman}. Notice that Haiman's conjecture is false in its full generality, see \cite[Theorem 1.8]{Nigro}.
\end{oss}

{\bf Acknowledgements:}
The authors would like to thank the organizers of the Workshop “Bruhat order: recent developments and open problems”, held
at the University of Bologna in April 2024 where part of our research was carried out.

\end{document}